\documentclass[a4paper]{article}
\usepackage{geometry}
\usepackage{amssymb}
\usepackage{latexsym}
\usepackage{amsmath}
\usepackage{indentfirst}
\usepackage{graphicx}
 \usepackage[colorlinks=true]{hyperref}
\usepackage{mathrsfs} 
\usepackage{chngcntr}
\usepackage{color}

\newtheorem{theorem}{Theorem}[section]

\newtheorem{lemma}{Lemma}[section]

\newtheorem{remark}{Remark}[section]

\newcommand{\N}{\mathbb{N}}

\newcommand{\R}{\mathbb{R}}
\newcommand{\F}{\mathcal{F}}

\newcommand{\M}{\mathcal{M}}
\numberwithin{equation}{section}

\newenvironment{proof}{\medskip\par\noindent{\bf Proof\/}:\quad}{\qquad
\raisebox{-0.5mm}{\rule{1.5mm}{1mm}}\vspace{6pt}}

\begin{document}
\title{Multiplicity and concentration of nontrivial nonnegative solutions for a fractional Choquard  equation with critical exponent\thanks{ This work is supported by NSFC (11361078,11661083,11771385), China. }}

\author{Shaoxiong Chen, Yue Li\thanks{Corresponding author:liyue9412@163.com. }, Zhipeng Yang\\
\footnotesize{\em Department\, of \,Mathematics,\, Yunnan\, Normal\, University,\, Kunming\, 650500 \,P.R. China }}

\date{} \maketitle
\noindent{\bf Abstract:}
In present paper, we study the fractional Choquard equation
$$\varepsilon^{2s}(-\Delta)^s u+V(x)u=\varepsilon^{\mu-N}(\frac{1}{|x|^\mu}\ast F(u))f(u)+|u|^{2^\ast_s-2}u$$
where $\varepsilon>0$ is a parameter, $s\in(0,1),$ $N>2s,$ $2^*_s=\frac{2N}{N-2s}$ and $0<\mu<\min\{2s,N-2s\}$. Under suitable assumption on $V$ and $f$, we prove this problem has a nontrivial nonnegative ground state solution. Moreover, we relate the number of nontrivial nonnegative solutions with the topology of the set where the potential attains its minimum values and their's concentration behavior.

\noindent{\bf Key Words:} Fractional Choquard equation; Ground state; Lusternik-Schnirelmann theory.\\
{\bf 2010 AMS Subject Classification:} 35P15, 35P30, 35R11.
\section{Introduction and the main results}
In this paper, we are interested in the existence, multiplicity and concentration behavior of the semi-classical solutions of the fractional Choquard equation
\begin{equation}\label{eq1.1}
\varepsilon^{2s}(-\Delta)^s u+V(x)u=\varepsilon^{\mu-N}(\frac{1}{|x|^\mu}\ast F(u))f(u)+|u|^{2^\ast_s-2}u, \,\,\ x\in \R^N
\end{equation}
where $\varepsilon >0$ is a parameter, $s\in(0,1),$ $N>2s,$ $2^*_s=\frac{2N}{N-2s},$ $0<\mu<\min\{2s,N-2s\}$ and $F(u)=\int^t_0f(\tau)d\tau$. The fractional Laplacian $(-\Delta)^s$ is defined by
\begin{equation*}
(-\Delta)^s\Psi(x)=C_{N,s}P.V.\int_{\R^N}\frac{\Psi(x)-\Psi(y)}{|x-y|^{N+2s}}dy,\ \ \Psi\in\mathcal{S}(\R^N),
\end{equation*}
where $P.V.$ stands for the Cauchy principal value, $C_{N,s}$ is a normalized constant, $\mathcal{S}(\R^N)$ is the Schwartz space of rapidly decaying functions, $s\in(0,1)$. As $\varepsilon$ goes to zero in \eqref{eq1.1}, the existence and asymptotic behavior of the solutions of the singularly perturbed equation \eqref{eq1.1} is known as the semi-classical problem. It was used to describe the transition between of Quantum Mechanics and Classical Mechanics.
\par
Our motivation to study \eqref{eq1.1} mainly comes from the fact that solutions $u(x)$ of \eqref{eq1.1} corresponding to standing wave solutions $\Psi(x,t)=e^{-iEt/\varepsilon}u(x)$ of the following time-dependent fractional Schr\"{o}dinger equation
\begin{equation}\label{eq1.2}
 i\varepsilon\frac{\partial\Psi}{\partial t}=\varepsilon^{2s}(-\Delta)^s\Psi+(V(x)+E)\Psi-(K(x)*|G(\Psi)|)g(\Psi)\ \  (x,t)\in{\R^N\times\R}
\end{equation}
where $i$ is the imaginary unit, $\varepsilon$ is related to the Planck constant. Equations of the type \eqref{eq1.2} was introduced by Laskin (see \cite{Laskin2000Physics,Laskin2002physics}) and come from an expansion of the Feynman path integral from Brownian-like to L\'{e}vy-like quantum mechanical paths. With variational methods, this kind equation has been studied widely, we refer to \cite{Valdinocibook,Nezza2012BDSM,Yang-Zhao18AML} and the references therein.
\par
When $s=1$, the equation \eqref{eq1.1} turns out to be the Choquard equation
\begin{equation}\label{eq1.3}
-\varepsilon^2\Delta u+V(x)u=\varepsilon^{\mu-N}(\frac{1}{|x|^\mu}\ast F(u))f(u)+|u|^{2^\ast-2}u\ \ \text{in}\ \R^N,
\end{equation}
The existence, multiplicity and concentration of solutions for \eqref{eq1.3} has been widely investigated.
On one hand, some people have studied the classical problem, namely $\varepsilon=1$ in \eqref{eq1.3}. When $V=1$ and $F(u)=\frac{|u|^q}{q}$, \eqref{eq1.3} covers in particular the Choquard-Pekar equation
\begin{equation}\label{eq1.4}
-\Delta u+u=(\int_{\R^N}\frac{1}{|x|^\mu}*|u|^qdy)|u|^{q-2}u\ \ \text{in}\ \R^N.
\end{equation}
The case $N=3$, $q=2$ and $\mu=1$ came from Pekar \cite{Pekar} in 1954 to describe the quantum mechanics of a polaron at rest. In 1976 Choquard used \eqref{eq1.4} to describe an electron trapped in its own hole, in a certain approximation to Hartree-Fock theory of one component plasma \cite{Lieb1976SAM}. In this context \eqref{eq1.4} is also known as the nonlinear Schr\"{o}dinger-Newton equation. By using critical point theory, Lions \cite{Lions80NA} obtained  the existence of infinitely many radialy symmetric solutions in $H^1(\R^N)$ and Ackermann \cite{Ackermann04MZ} prove the existence of infinitely many geometrically distinct weak solutions for a general case. For the properties of the ground state solutions, Ma and Zhao \cite{Ma-Zhao10ARMA} proved that every positive solution is radially symmetric and monotone decreasing about some point for the generalized Choquard equation \eqref{eq1.4} with $q\geq2$. Later, Moroz and Van Schaftingen \cite{Moroz-Schaftingen13JFA,Moroz-Schaftingen15TAMC} eliminated this
restriction and showed the regularity, positivity and radial symmetry of the ground
states for the optimal range of parameters, and also derived that these solutions
decay asymptotically at infinity.

\par
On the other hand, some people have focused on the semiclassical problem, namely, $\varepsilon\rightarrow0$ in \eqref{eq1.3}. The question of the existence of semiclassical solutions for the non-local problem \eqref{eq1.3} has been posed in \cite{Ambrosetti-Malchiodi07CM}. Note that if $v$ is a solution of \eqref{eq1.3} for $x_0\in\R^N$, then $u=v(\varepsilon x+x_0)$ verifies
\begin{equation}\label{eq1.5}
-\Delta u+V(\varepsilon x+x_0)u=(\int_{\R^N}\frac{G(u(y))}{|x-y|^\mu}dy)g(u)\ \ \text{in}\ \R^N,
\end{equation}
which means some convergence of the family of solutions to a solution $u_0$ of the limit problem
\begin{equation}\label{eq1.6}
-\Delta u+V(x_0)u=(\int_{\R^N}\frac{G(u(y))}{|x-y|^\mu}dy)g(u)\ \ \text{in}\ \R^N.
\end{equation}
For this case when $N=3, \mu=1$ and $G(u)=|u|^2$, Wei and Winter \cite{Wei-Winter09JMP} constructed families of solutions by a Lyapunov-Schmidt-type reduction when $\inf V>0$. This method of construction depends on the existence, uniqueness and non-degeneracy up to translations of the positive solution of the limiting equation \eqref{eq1.6}, which is a difficult problem that has only been fully solved in the case when $N=3, \mu=1$ and $G(u)=|u|^2$.
Moroz and Van Schaftingen \cite{Moroz-Schaftingen15CVPDE} used variational methods to develop a novel non-local penalization technique to show that equation \eqref{eq1.3} with $G(u)=|u|^q$ has a family of solutions concentrated at the local minimum of $V$, with $V$ satisfying some additional assumptions at infinity.
In addition, Alves and Yang \cite{Alves-Yang16PRSE} investigated the multiplicity and concentration behaviour of solutions for a quasi-linear Choquard equation via the penalization method.
Very recently, in an interesting paper, Alves et al. \cite{Alves-Yang17JDE} study \eqref{eq1.4} with a critical growth, they consider the critical problem with both linear potential and nonlinear potential, and showed the existence, multiplicity and concentration behavior of solutions when the linear potential has a global minimum or maximum.
\par
On the contrary, the results about fractional Choquard equation \eqref{eq1.1} are relatively few. Recently, d'Avenia, Siciliano and Squassina \cite{d'AveniaMMMAS15} studied the existence, regularity and asymptotic of the solutions for the following fractional Choquard equation
\begin{equation}\label{eq1.7}
(-\Delta)^su+\omega u=(\int_{\R^N}\frac{|u(y)|^q}{|x-y|^\mu}dy)|u|^{q-2}u\ \ \text{in}\ \R^N,
\end{equation}
where $\omega>0$, $\frac{2N-\mu}{N}<q<\frac{2N-\mu}{N-2s}$. Shen, Gao and Yang \cite{Shen-Gao-YangMMAS16} obtain the existence of ground states for \eqref{eq1.7} with general nonlinearities by using variational methods. Chen and Liu \cite{Chen-Liu16Non} studied \eqref{eq1.7} with nonconstant linear potential and proved the existence of ground states without any symmetry property. For critical problem, Wang and Xiang \cite{Wang-Xiang16EJDE} obtain the existence of infinitely many nontrivial solutions and the Brezis-Nirenberg type results can be founded in \cite{Sreenadh17NODEA}. For other existence results we refer to \cite{Miyagaki17NA,Bhattarai17JDE,Gao-Tang-Chen18ZMAP,Guo-Hu18MMAS,Ma-Zhang17NA,Wang-Yang18BVP,Zhang-Wu18JMAA} and the references therein.
\par
For the concentration behavior of solutions, we note that the only works concerning the concentration behavior of solutions come from \cite{yang,Zhang-Wang-Zhang19CPAA}.
Assuming the global condition on $V\in C(\R^N,\R)$:
\begin{itemize}
\item[$(V_0)$] $0<V_0:=\inf\limits_{x\in\R^N}V(x)<\liminf\limits_{|x|\rightarrow\infty}V(x):=V_{\infty}<+\infty$,
\end{itemize}
which is firstly introduced by Rabinowitz \cite{Rabinowitz1992} in the study of the nonlinear Schr\"{o}dinger equations. By using the method of Nehari manifold developed by Szulkin and Weth \cite{Szulkin2010}, Zhang, Wang and Zhang in \cite{Zhang-Wang-Zhang19CPAA} obtained the multiplicity and concentration of positive solutions for the following  fractional Choquard equation
\begin{equation}\label{eq1.8}
\varepsilon^{2s}(-\Delta)^su+V(x)u=\varepsilon^{\mu-3}(\int_{\R^3}\frac{|u(y)|^{2^*_{\mu,s}}+F(u(y))}{|x-y|^\mu}dy)(|u|^{2^*_{\mu,s}-2}u+\frac{1}{2^*_{\mu,s}}f(u))\ \ \text{in}\ \R^3,
\end{equation}
where $\varepsilon>0$, $0<\mu<3$, $F$ is the primitive function of $f$. Different to the global condition $(V_0)$, Yang in \cite{yang} establish the existence and concentration of positive solutions for the fractional Choquard equation \eqref{eq1.8} when the potential function $V\in C(\R^3,\R)$ satisfies the following local conditions \cite{Pino1996CVPDE}:
\begin{itemize}
\item[$(V_1)$] There is constant $V_0>0$ such that $V_0=\inf\limits_{x\in\R^3}V(x)$.
\item[$(V_2)$] There is a bounded domain $\Omega$ such that
\begin{equation*}
V_0<\min_{\partial\Omega}V.
\end{equation*}
\end{itemize}
\par
Note that in \eqref{eq1.8}, the critical term is involved in the convolution-type nonlinearity, which is totally different from our problem \eqref{eq1.1}. It is natural to ask how about the concentration behavior of solutions of \eqref{eq1.1} as $\varepsilon\to 0^+$? And how about the influence of the potential on the multiplicity of solutions? However, to the best of our knowledge, it seems that these two problems were not considered in literatures before. In this paper, we are concerned with the multiplicity and concentration property of nontrivial nonnegative solutions to \eqref{eq1.1}, and we will give some answers to the above questions.
\par
Concerning the continuous function $f\in C(\R,\R)$, we assume that $f(t)=0$ for $t<0$ and satisfies the following conditions:
\begin{itemize}
\item[$(f_1)$] $\lim\limits_{t\rightarrow 0}\frac{f(t)}{t}=0$.
\item[$(f_2)$] $\exists$ $ q\in(\frac{2N-\mu}{N},\frac{2N-\mu}{N-2s})$ such that $\lim\limits_{t\rightarrow\infty}\frac{f(t)}{t^{q-1}}=0.$
\item[$(f_3)$] $\frac{f(t)}{t}$ is increasing for every $t>0$.
\item[$(f_4)$] $\exists$ $\sigma\in(q_{_{N}},\frac{2N-\mu}{N-2s}),$ $C>0$ $s.t.$ $f(t)\geq ct^{\sigma-1}$ for all $t\in\R^+,$ where $q_{_{N}}=\max\{\frac{2N-2s}{N-2s},\frac{N+2s}{N-2s}\}.$
\end{itemize}
Then we state our main result as follows.
\begin{theorem}\label{Thm1.1}
Suppose $(V_0)$ hold and $f$ satisfies $(f_1)-(f_4)$. Then there exists an $\varepsilon^*>0$ such that for any $\varepsilon\in(0,\varepsilon^*),$ the problem \eqref{eq1.1} possesses a nontrivial nonnegative ground state solution.
\end{theorem}
In order to describe the multiplicity, we first recall that, if $Y$ is a closed subset of a topological space $X$, the Ljusternik-Schnirelmann category $cat_{X}Y$ is the least number of closed and contractible sets in $X$ which cover $Y$.
Then we have our second result as follows.
\begin{theorem}\label{Thm1.2}
Suppose $(V_0)$ hold and $f$ satisfies $(f_1)-(f_4)$. Then for any $\delta>0$, there exists $\varepsilon_{\delta}>0$ such that for any $\varepsilon\in(0,\varepsilon_{\delta}),$ the problem \ref{eq1.1} has at least $cat_{\Lambda_{\delta}}(\Lambda)$ nontrivial nonnegative solutions. Moreover, if $u_{\varepsilon}$ denotes one of these solutions and $x_{\varepsilon}\in{\R^N}$ is its global maximum, then
\begin{equation*}
\lim\limits_{\varepsilon\rightarrow0}V(x_{\varepsilon})=V_0,
\end{equation*}
where $\Lambda:=\{x\in\R^N:V(x)=V_0\}$ and $\Lambda_{\delta}:=\{x\in\R^N:d(x,\Lambda)\leq\delta\}.$
\end{theorem}

\par
We shall use the method of Nehari manifold, concentration compactness principle and category theory to prove the main results. There are some difficulties in proving our theorems. The first difficulty is that the nonlinearity $f$ is only continuous, we need to prove the new Brezis-Lieb type Lemma for this kind of nonlinearity. The second one is the lack of compactness of the embedding of $H^s(\R^3)$ into the space $L^{2^*_s}(\R^3)$. We shall borrow the idea in \cite{Alves-Yang17JDE,Cassani-Zhang18ANA} to deal with the difficulties brought by the critical exponent. However, we require some new estimates, which are complicated because of the appearance of fractional Laplacian and the convolution-type nonlinearity.
\par
This paper is organized as follows. In section 2, besides describing the functional setting to study problem \eqref{eq1.1}, we give some preliminary Lemmas which will be used later. In section 3, we prove problem \eqref{eq1.1} has a ground state solution. Finally, we show the multiple of nontrivial nonnegative solutions and investigate its concentration behavior, which completes the proof Theorem \ref{Thm1.1}.
\par
\textbf{Notation.}~In this paper we make use of the following notations.
\begin{itemize}
\item[$\bullet$] For any $R>0$ and for any $x\in\R^N$, $B_{R}(x)$ denotes the ball of radius $R$ centered at $x$.
\item[$\bullet$]  $L^p(\mathbb{R}^N)$, $1\leq p<+\infty$ denotes the Lebesgue space with  the norm $\|u\|_p=\|u\|_{L^p(\R^N)}=(\int_{\mathbb{R}^N}|u|^pdx)^{\frac{1}{p}}$.
\item[$\bullet$] The letters $C,C_i$ stand for positive constants (possibly different from line to line).
\item[$\bullet$]  "$\rightarrow$" for the strong convergence and "$\rightharpoonup$" for the weak convergence.
\item[$\bullet$]  $u^+=\max\{u,0\}$ and $u^-=\min\{u,0\}$ denote the positive part and the negative part of a function $u$, respectively.
\end{itemize}

\section{Functional Setting}
Firstly, fractional Sobolev spaces are the convenient setting for our problem, so we will give some skrtchs of the fractional order Sobolev spaces and the complete introduction can be found in \cite{Nezza2012BDSM}. We recall that, for any $s\in(0,1)$, the fractional Sobolev space $H^s(\R^N)=W^{s,2}(\R^N)$ is defined as follows:
\begin{equation*}
H^s(\R^N)=\{u\in L^2(\R^N):\int_{\R^N}\big(|\xi|^{2s}|{\F(u)}|^2+|{\F(u)}|^2\big)d\xi<\infty\},
\end{equation*}
whose norm is defined as
\begin{equation*}
\|u\|^2_{H^{s}(\R^N)}=\int_{\R^N}\big(|\xi|^{2s}|{\F(u)}|^2+|{\F(u)}|^2\big)d\xi,
\end{equation*}
where $\F$ denotes the Fourier transform. We also define the homogeneous fractional Sobolev space $\mathcal{D}^{s,2}(\mathbb{R}^N)$ as the completion of $\mathcal{C}_0^\infty(\mathbb{R}^N)$ with respect to the norm
\begin{equation*}
\|u\|_{\mathcal{D}^{s,2}(\R^N)}:=\bigg(\iint_{\mathbb{R}^N\times \mathbb{R}^N}\frac{|u(x)-u(y)|^2}{|x-y|^{N+2s}}dxdy\bigg)^{\frac{1}{2}}=[u]_{H^s(\R^N)}.
\end{equation*}
\par
The embedding $\mathcal{D}^{s,2}(\R^N)\hookrightarrow L^{2^*_s}(\R^N)$ is continuous and for any $s\in(0,1)$, there exists a best constant $S_s>0$ such that
\begin{equation*}
S_s:=\inf_{u\in\mathcal{D}^{s,2}(\R^N)}\frac{\|u\|^2_{\mathcal{D}^{s,2}(\R^N)}}{\|u\|_{L^{2^*_s}(\R^N)}^2}
\end{equation*}
According to \cite{Cotsiolis-Tavoularis04JMAA}, $S_s$ is attained by
\begin{equation}\label{eq2.1}
u_0(x)=C\big(\frac{b}{b^2+|x-a|^2}\big)^{\frac{N-2s}{2}},\ \ x\in\R^N,
\end{equation}
where $C\in\R$, $b>0$ and $a\in\R^N$ are fixed parameters.
\par
The fractional laplacian, $(-\Delta)^s u$, of a smooth function $u:\R^N\rightarrow\R$, is defined by
$$\F((-\Delta)^s u)(\xi)=|\xi|^{2s}\F(u)(\xi),\ \  \xi\in\R^N.$$
Also $(-\Delta)^s u$ can be equivalently represented \cite{Nezza2012BDSM} as
$$(-\Delta)^s u(x)=-\frac{1}{2}C(N,s)\int_{\R^N}\frac{u(x+y)+u(x-y)-2u(x)}{|y|^{N+2s}}dy,\ \forall x\in\R^N$$
where
$$C(N,s)=\bigg(\int_{\R^N}\frac{(1-cos\xi_1)}{|\xi|^{N+2s}}d\xi\bigg)^{-1},\ \xi=(\xi_1,\cdots,\xi_N).$$
Also, by the Plancherel formular in Fourier analysis, we have
$$[u]^2_{H^{s}(\R^N)}=\frac{2}{C(N,s)}\|(-\Delta)^{\frac{s}{2}}u\|^2_{L^2(\R^N)}.$$
For convenience, we will omit the normalization constant in the following. As a consequence, the norms on $H^s(\R^N)$ defined below
$$\aligned
&u\longmapsto\bigg(\int_{\R^N}|u|^2dx+\iint_{\R^N\times\R^N}\frac{|u(x)-u(y)|^2}{|x-y|^{N+2s}}dxdy\bigg)^{\frac{1}{2}};\\
&u\longmapsto\bigg(\int_{\R^N}(|\xi|^{2s}|{\F(u)}|^2+|{\F(u)}|^2)d\xi\bigg)^{\frac{1}{2}};\\
&u\longmapsto\bigg(\int_{\R^N}|u|^2dx+\|(-\Delta)^{\frac{s}{2}}u\|^2_{L^2(\R^N)}\bigg)^{\frac{1}{2}}.\\
\endaligned $$
are equivalent.
\par
Making the change of variable $x\mapsto \varepsilon x$, we can rewrite the equation \eqref{eq1.1} as the following equivalent form
\begin{equation}\label{eq2.3}
(-\Delta)^su+V(\varepsilon x)u=(\frac{1}{|x|^\mu}\ast F(u))f(u)+|u|^{2^\ast_s-2}u\ \ \text{in}\ \R^N,
\end{equation}
If $u$ is a solution of the equation \eqref{eq2.3}, then $v(x):=u(\frac{x}{\varepsilon})$ is a solution of the equation \eqref{eq1.1}. Thus, to study the equation \eqref{eq1.1}, it suffices to study the equation \eqref{eq2.3}.
In view of the presence of potential $V(x)$, we introduce the subspace
\begin{equation*}
H_\varepsilon=\bigg\{u\in H^s(\mathbb{R}^N):\int_{\mathbb{R}^N}V(\varepsilon x)u^2dx<+\infty\bigg\},
\end{equation*}
which is a Hilbert space equipped with the inner product
\begin{equation*}
(u,v)_{H_\varepsilon}=\int_{\mathbb{R}^N}(-\Delta)^{\frac{s}{2}} u(-\Delta)^{\frac{s}{2}} vdx+\int_{\mathbb{R}^N}V(\varepsilon x)uvdx,
\end{equation*}
and the norm
\begin{equation*}
\|u\|_{H_\varepsilon}^2=\int_{\mathbb{R}^N}|(-\Delta)^{\frac{s}{2}} u|^2dx+\int_{\mathbb{R}^N}V(\varepsilon x)u^2dx.
\end{equation*}
We denote $\|\cdot\|_{H_\varepsilon}$ by $\|\cdot\|_\varepsilon$ in the sequel for convenience. The energy functional corresponding to equation \eqref{eq2.3} is $$E_{\varepsilon}(u)=\frac{1}{2}\|u\|^2_{\varepsilon}-\frac{1}{2}\int_{\R^N}(\frac{1}{|x|^{\mu}}\ast F(u))F(u)dx-\frac{1}{2^*_s}\int_{\R^N}|u|^{2^*_s}dx.$$
Since we are interested in the nontrivial nonnegative solutions, we consider the following functional
$$J_{\varepsilon}(u)=\frac{1}{2}\|u\|^2_{\varepsilon}-\frac{1}{2}\int_{\R^N}(\frac{1}{|x|^{\mu}}\ast F(u^+))F(u^+)dx-\frac{1}{2^*_s}\int_{\R^N}|u^+|^{2^*_s}dx.$$
Moreover, $J_{\varepsilon}(u)\in C^1(H^s,\R^N),$
\begin{equation*}
\begin{split}
\langle J'_{\varepsilon}(u),\varphi\rangle=&\int\int_{\R^{N}\times\R^N}\frac{u(x)-u(y)}{|x-y|^{N+2s}}(\varphi(x)-\varphi(y))dxdy+\int_{\R^N}V(\varepsilon x)u\varphi dx\\
&-\int_{\R^N}(\frac{1}{|x|^{\mu}}\ast F(u^+))f(u^+)\varphi dx-\int_{\R^N}|u^+|^{2^*_s-2}u\varphi dx.
\end{split}
\end{equation*}
\par
We collect the following useful result.
\begin{lemma}\label{Lem2.1}
Let $s\in(0,1)$ and $N>2s.$ Then there exists a sharp constant $C_\ast=C(N,s)>0$ such that for any $u\in H^s(\R^N)$ $$\|u\|^2_{L^{2^{\ast}_s}(\R^N)}\leq C^{-1}_{*}[u]^2_{H^s(\R^N)}.$$
Moreover $H^s(\R^N)$ is continuously embedded in $L^q(\R^N)$ for any $q\in[2,2^*_s]$ and compactly in $L^q_{loc}(\R^N)$ for any $q\in [2,2^*_s).$
\end{lemma}
\begin{lemma}\label{Lem2.2}
Let $N>2s,$ If $\{u_n\}$ is a bounded sequence in $H^s(\R^N)$ and if $$\lim_{n\rightarrow\infty}\sup_{y\in \R^N}\int_{B_R(y)}|u_n|^2dx=0$$
where $R>0,$ then $u_n\rightarrow 0$ in $L^t(\R^N)$ for all $t\in[2,2^*_s).$
\end{lemma}
\begin{lemma}\label{Lem2.3}
Let $t,r>1$ and $0<\mu<N$ such that $\frac{1}{r}+\frac{\mu}{N}+\frac{1}{t}=2.$ Let $f\in L^r(\R^N)$ and $h\in L^t(\R^N).$ Then there exists a sharp constant $C(r,N,\mu,t)>0,$ independent of $f \text{and } h,$ such that $$\int_{\R^N}\int_{\R^N}\frac{f(x)h(y)}{|x-y|^{\mu}}dxdy\leq C(r,N,\mu,t)\|f\|_{L^r(\R^N)}\|h\|_{L^t(\R^N)}.$$
\end{lemma}
\begin{lemma}\label{Lem2.4}
The space $H_{\varepsilon}$ is continuously embedded into $H^s(\R^N).$ Therefore, $H_{\varepsilon}$ is continously embedded into $L^r(\R^N)$ for any $r\in[2,2^*_s]$ and compactly embedded into $L^r_{loc}(\R^N)$ for any $r\in[2,2^*_s).$
\end{lemma}
\begin{lemma}\cite{Palatucci}\label{Lem2.5}Let $u\in \mathcal{D}^{s,2}(\mathbb{R}^N)$, $\varphi\in C_0^\infty(\mathbb{R}^N)$ and for each $r>0,~\varphi_r(x)=\varphi(\frac{x}{r})$. Then
\begin{equation*}
u\varphi_r\rightarrow 0~\text{in}~\mathcal{D}^{s,2}(\mathbb{R}^N)~\text{as}~r\rightarrow0.
\end{equation*}
If, in addition, $\varphi\equiv 1$ in a neighbourhood of the origin, then
\begin{equation*}
u\varphi_r\rightarrow u~\text{in}~\mathcal{D}^{s,2}(\mathbb{R}^N)~\text{as}~r\rightarrow+\infty.
\end{equation*}
\end{lemma}

\section{Ground state solution}
\begin{lemma}\label{Lem3.1}
$J_{\varepsilon}$ has a mountain pass geometry, that is
\begin{itemize}
\item[$(i)$] There exists $\alpha,\rho >0$ such that $J_{\varepsilon}(u)\geq\alpha$ for any $u\in H_{\varepsilon}$ which $\|u\|_{\varepsilon}=\rho.$
\item[$(ii)$] There exists $e\in H_{\varepsilon}$ with $\|e\|_{\varepsilon}>\rho$ such that $J_{\varepsilon}(e)<0.$
\end{itemize}
\end{lemma}
\begin{proof}
In order to show this, we argue as in Lemma2.2 in \cite{Ambrosio2017Multiplicity}.  From $(f_1)$ and $(f_2),$ it follows that for any $\xi>0$ there exists $C_{\xi}>0$ such that
\begin{equation}\label{eq3.1}
f(t)\leq\xi|t|+C_{\xi}|t|^{q-1},\,\,\,\ F(t)\leq\xi|t|^2+C_{\xi}|t|^{q}.
\end{equation}
By \eqref{eq3.1} and Lemma \ref{Lem2.3}, we get
\begin{equation}\label{eq3.2}
\begin{split}
|\int_{\R^N}(\frac{1}{|x|^{\mu}}\ast F(u^+))F(u^+)dx|&\leq C\|F(u)\|_{L^t(\R^N)}\|F(u)\|_{L^t(\R^N)}\\
&\leq C(\int_{\R^N}(|u|^2+|u|^q)^tdx)^{\frac{2}{t}}
\end{split}
\end{equation}
where $t=\frac{2N}{2N-\mu}.$
Since $q\in(\frac{2N-\mu}{N},\frac{2N-\mu}{N-2s}),$ we can see that $tq\in(2,2^*_s),$ and from Lemma \ref{Lem2.4}, we have
\begin{equation}\label{eq3.3}
(\int_{\R^N}(|u|^2+|u|^q)^tdx)^{\frac{2}{t}}\leq C(\|u\|^2_{\varepsilon}+\|u\|^q_{\varepsilon})^2.
\end{equation}
Taking into account \eqref{eq3.2} and \eqref{eq3.3} we can deduce that
\begin{equation}\label{eq3.3.1}
\int_{\R^N}(\frac{1}{|x|^{\mu}}\ast F(u^+))F(u^+)dx+\frac{1}{2^*_s}\int_{\R^N}|u^+|^{2^*_s}dx\leq C(\|u\|^4_{\varepsilon}+\|u\|^{2q}_{\varepsilon}+\|u\|^{2^*_s}_{\varepsilon}).
\end{equation}
As a consequence
$$J_{\varepsilon}(u)\geq\frac{1}{2}\|u\|^2_{\varepsilon}-C(\|u\|^4_{\varepsilon}+\|u\|^{2q}_{\varepsilon}+\|u\|^{2^*_s}_{\varepsilon}).$$
We can see that $(i)$ holds.
\\
Fix a positive function $u_0\in H_{\varepsilon}(\R^N)\setminus\{0\}$ and $u_0>0,$ we set
$$h(t)=\frac{1}{2}\int_{\R^N}(\frac{1}{|x|^{\mu}}\ast F(\frac{tu_0}{\|u_0\|_{\varepsilon}}))F(\frac{tu_0}{\|u_0\|_{\varepsilon}})dx \,\,\ \text{for} \,\,\ t>0.$$
By $(f_3),$ we have $$F(u)=\int^1_0f(tu)udt=\int^1_0\frac{f(tu)}{tu}tu^2dt\leq\int^1_0f(u)tudx=\frac{1}{2}f(u)u \,\,\ \,\,\ \text{for} \,\,\ u>0.$$
Hence,
\begin{equation}\label{eq3.4}
\begin{split}
h'(t)&=\int_{\R^N}(\frac{1}{|x|^{\mu}}\ast F(\frac{tu_0}{\|u_0\|_{\varepsilon}}))f(\frac{tu_0}{\|u_0\|_{\varepsilon}})\frac{u_0}{\|u_0\|_{\varepsilon}}dx\\
&=\frac{4}{t}\int_{\R^N}\frac{1}{2}(\frac{1}{|x|^{\mu}}\ast F(\frac{tu_0}{\|u_0\|_{\varepsilon}}))\frac{1}{2}f(\frac{tu_0}{\|tu_0\|_{\varepsilon}})\frac{tu_0}{\|u_0\|_{\varepsilon}}dx\\
&\geq\frac{4}{t}h(t).
\end{split}
\end{equation}
Integrating \eqref{eq3.4} on $[1,t\|u_0\|_{\varepsilon}]$ with $t>\frac{1}{\|u_0\|_{\varepsilon}},$ we find
$$h(t\|u_0\|_{\varepsilon})\geq h(1)(t\|u_0\|_{\varepsilon})^4$$
which gives
$$\frac{1}{2}\int_{\R^N}(\frac{1}{|x|^{\mu}}\ast F(tu_0))F(tu_0)dx\geq\frac{1}{2}\int_{\R^N}(\frac{1}{|x|^{\mu}}\ast F(\frac{u_0}{\|u_0\|_{\varepsilon}}))F(\frac{u_0}{\|u_0\|_{\varepsilon}})dx\|u_0\|^4_{\varepsilon}t^4.$$
Therefore, we have
\begin{equation*}
\begin{split}
J_{\varepsilon}(tu_0)&=\frac{t^2}{2}\|u_0\|_{\varepsilon}^2-\frac{1}{2}\int_{\R^N}(\frac{1}{|x|^{\mu}}\ast F(tu_0))F(tu_0)dx-\frac{t^{2^*_s}}{2^*_s}\int_{\R^N}|u_0|^{2^*_s}dx\\
&\leq C_1t^2-C_2t^4
\end{split}
\end{equation*}
for $t>\frac{1}{\|u_0\|_{\varepsilon}}.$
Taking $e=tu_0$ with $t$ sufficiently large, we can see that $(ii)$ holds.
\end{proof}\\
Let us denote by $\mathcal{S}_\varepsilon$ the unitary sphere in $H_\varepsilon.$
\begin{lemma}\label{Lem3.2}
For each $u\in X^+_{\varepsilon}:=\{u\in H_{\varepsilon}:u^+(x)\neq 0\}$ and $t>0,$ set $h_u(t):=J_{\varepsilon}(tu).$
\begin{itemize}
\item[$(i)$]  Then there exists an unique $t_u>0$ such that $h_u(t_u)=\max\limits_{t\geq0}h_u(t)=\max\limits_{t\geq0}J_{\varepsilon}(tu),$ $h'_{u}(t_u)=0,$ $h'_{u}(t)>0$ in $(0,t_u),$ $h'_u(t)<0$ in $(t_u,+\infty)$ and $tu\in \mathcal{N}_{\varepsilon}$ if and only if $t=t_u,$ where $\mathcal{N}_{\varepsilon}=\{u\in X^+_{\varepsilon}:\langle J'_{\varepsilon}(u),u\rangle=0\}.$
\item[$(ii)$]  There is $\kappa>0$ independent on $u,$ such that $t_u\geq\kappa$ for all $u\in\mathcal{S}_\varepsilon.$ Moreover, for any compact set $E\subset\mathcal{S}_\varepsilon,$ there is a $C_{E}>0$ such that $t_u\leq C_{E}$ for all $u\in E.$
\end{itemize}
\end{lemma}
\begin{proof}
$(i)$ For every $u\in X^+_{\varepsilon},$ from Lemma \eqref{Lem3.1} we know that $h_u(0)=0,$ $h_u(t)>0$ for $t>0$ small enough and $\lim\limits_{t\rightarrow+\infty}h_u(t)=-\infty.$ Hence, there exists a $t_u>0$ such that $h_u(t_u)=\max\limits_{t\geq0}h_u(t)$ and $h'_u(t_u)=0.$ Notice that
$$h'_u(t)=0\Leftrightarrow tu\in\mathcal{N}_{\varepsilon}\Leftrightarrow \|u\|^2_{\varepsilon}=\int_{\R^N}(\frac{1}{|x|^{\mu}}\ast\frac{F(tu^+)}{t})f(tu^+)u^+dx+t^{2^*_s-2}\int_{\R^N}|u^+|^{2^*_s}dx.$$
From $(f_3)$ we know $t\mapsto f(t)$ and $t\mapsto\frac{F(t)}{t}$ are increasing for all $t>0.$ Hence, we get the uniqueness of a such $t_u$ and $(i)$ is completed.\\
$(ii)$ Let $u\in\mathcal{S}_\varepsilon.$ By $t_uu\in\mathcal{N}_\varepsilon$ and \eqref{eq3.3.1} we have
$$t_u^2=\int_{\R^N}(\frac{1}{|x|^{\mu}}\ast F(t_uu^+))f(t_uu^+)t_uu^+dx+\int_{\R^N}|t_uu^+|^{2^*_s}dx\leq C(t_u^4+t_u^{2q}+t_u^{2^*_s}).$$
So, there exists $\kappa>0$ independent of $u$, such that $t_u\geq\kappa.$ Let, $\alpha\in(2,2^*_s),$ $\alpha\leq4$ then $\frac{2}{\alpha}\geq\frac{1}{2}.$ We can infer that
$$F(t)\leq\frac{1}{2}f(t)t\leq\frac{2}{\alpha}f(t)t, \,\,\ \forall~t\geq0.$$
For any $v\in\mathcal{N}_\varepsilon,$ we have
\begin{equation}\label{eq3.3.2}
\begin{split}
J_\varepsilon(v)&= J_{\varepsilon}(v)-\frac{1}{\alpha}\langle J'_{\varepsilon}(v),v\rangle\\
&=(\frac{1}{2}-\frac{1}{\alpha})\|v\|_{\varepsilon}^2-\frac{1}{2}\int_{\R^N}(\frac{1}{|x|^{\mu}}\ast F(v^+))(F(v^+)-\frac{2}{\alpha}f(v^+)v^+)dx+(\frac{1}{\alpha}-\frac{1}{2^*_s})\int_{\R^N}|v^+|^{2^*_s}dx\\
&\geq(\frac{1}{2}-\frac{1}{\alpha})\|v\|_{\varepsilon}^2.
\end{split}
\end{equation}
If $E\subset\mathcal{S}_\varepsilon$ is a compact set and $u_n\subset E$ such that $t_{u_n}\rightarrow\infty,$ up to subset $u_n\rightarrow u$ in $H_\varepsilon$ and $J_\varepsilon(t_{u_n}u_n)\rightarrow-\infty.$ Taking $v_n=t_{u_n}u_n\in\mathcal{N}_\varepsilon$ in \eqref{eq3.3.2}, we can see that
$$0<\frac{1}{2}-\frac{1}{\alpha}\leq\frac{J_\varepsilon(t_{u_n}u_n)}{t_{u_n}^2}\leq0 \,\,\ as \,\,\ n\rightarrow\infty.$$
which gives a contradiction.
\end{proof}
\par
Define the mappings $\hat{n}_\varepsilon:H_\varepsilon\backslash\{0\}\rightarrow\mathcal{N}_\varepsilon$ and $n_\varepsilon:=\mathcal{S}_\varepsilon\rightarrow\mathcal{N}_\varepsilon$ by set
$$\hat{n}_\varepsilon(u):=t_uu \,\,\ \text{ and } \,\,\ n_\varepsilon:=\hat{n}_\varepsilon|_{\mathcal{S}_\varepsilon}.$$
We can apply \cite[Proposition8, Proposition9 and Corollary10 ]{Szulkin2010} to deduce the follow Lemma.
\begin{lemma}\label{lem3.3.1}
Suppose that $(V_0)$ and $(f_1)-(f_4),$ then
\begin{itemize}
\item[$(a)$]  The mapping $\hat{n}_\varepsilon$ is continuous and $n_\varepsilon$ is a homeomorphism between $\mathcal{S}_\varepsilon$ and $\mathcal{N}_\varepsilon.$ Moreover, $n_\varepsilon^{-1}(u)=\frac{u}{\|u\|_\varepsilon}.$
\item[$(b)$]  We define the maps $\hat{\psi}_\varepsilon:H_\varepsilon\backslash\{0\}\rightarrow\R$ by $\hat{\psi}_\varepsilon(u):=J_\varepsilon(\hat{n}_\varepsilon(u)).$ Then $\hat{\psi}_\varepsilon\in C^1(H_\varepsilon\backslash\{0\},\R)$ and
    $$\langle\hat{\psi}_\varepsilon'(u),v\rangle=\frac{\|\hat{n}_\varepsilon(u)\|_\varepsilon}{\|u\|_\varepsilon}\langle J'_\varepsilon(\hat{n}_\varepsilon(u),v)\rangle$$
    for every $u\in H_\varepsilon\backslash\{0\}$ and $v\in H_\varepsilon.$
\item[$(c)$]  We define the maps $\psi:\mathcal{S}_\varepsilon\rightarrow\R$ by $\psi_\varepsilon:=\hat{\psi}|_{\mathcal{S}_\varepsilon}.$ Then
    $\psi_\varepsilon\in C^1(\mathcal{S}_\varepsilon,\R)$ and $\langle\psi'_\varepsilon(u),v\rangle=\|n_\varepsilon(u)\|_\varepsilon\langle J_\varepsilon'(n_\varepsilon(u)),v\rangle$ for any $v\in T_u\mathcal{S}_\varepsilon.$
\item[$(d)$]  If $\{u_n\}$ is a $(PS)_d$ sequence for $\psi_\varepsilon,$ then $\{n_\varepsilon(u_n)\}$ is a $(PS)_d$ sequence for $J_\varepsilon.$ Moreover, if $\{u_n\}\subset\mathcal{N}_\varepsilon$ is a bounded $(PS)_d$ sequence for $\psi_\varepsilon,$ then $\{n_\varepsilon^{-1}(u_n)\}$ is a $(PS)_d$ sequence for the functional $\psi_\varepsilon.$
\item[$(e)$]  $u$ is a critical point of $\psi_\varepsilon$ if and only if $n_\varepsilon(u)$ is a nontrivial critical point for $J_\varepsilon.$ Moreover, the corresponding critical values coincide and
    $$\inf_{u\in\mathcal{S}_\varepsilon}\psi_\varepsilon(u)=\inf_{u\in\mathcal{N}_\varepsilon}J_\varepsilon(u).$$
\end{itemize}
\end{lemma}
\begin{remark}\label{rem3.2}
As in \cite{Szulkin2010}, we have the following minimax characterization:
$$c_\varepsilon=\inf_{u\in\mathcal{N}_\varepsilon}J_\varepsilon(u)=\inf_{u\in H_\varepsilon\backslash\{0\}}\max_{t>0}J_\varepsilon(tu)=\inf_{u\in\mathcal{S}_\varepsilon}\max_{t>0}J_\varepsilon(tu).$$
\end{remark}
\par
Recall that $\{u_n\}\subset H_{\varepsilon}$ is called a $(C)_d$ sequence of $J_{\varepsilon}$ if $J_{\varepsilon}(u_n)\rightarrow d$ and $ (1+\|u_n\|_{\varepsilon})J'_{\varepsilon}(u_n)\rightarrow0,$ where $d>0$ $J_{\varepsilon}$ satisfies the $(C)_d$ condition if every $(C)_d$ sequence of $J_{\varepsilon}$ has a convergent subsequence. And $J_{\varepsilon}$ satisfies the $(PS)_d$ condition if $J_{\varepsilon}$ satisfies the $(C)_d$ condition. Next, we give some properties of $(PS)_d$ sequence of $J_{\varepsilon}.$
\begin{lemma}\label{Lem3.3}
Let $\{u_n\}\subset H_{\varepsilon}$ is a $(PS)_d$ sequence of $J_{\varepsilon},$ then $\{u_n\}$ is bounded in $H^s(\R^N)$ and $\{\|u_n^-\|_{\varepsilon}\}=o_n(1).$
\end{lemma}
\begin{proof}
Let, $\alpha\in(2,2^*_s),$ $\alpha\leq4$ then $\frac{2}{\alpha}\geq\frac{1}{2}.$ We can infer that
$$F(t)\leq\frac{1}{2}f(t)t\leq\frac{2}{\alpha}f(t)t, \,\,\ \forall t\geq0.$$
Since $\{u_n\}$ is a $(PS)_d$ sequence of $J_{\varepsilon},$ we have
\begin{equation*}
\begin{split}
d+1+\|u_n\|_{\varepsilon}&\geq J_{\varepsilon}(u_n)-\frac{1}{\alpha}\langle J'_{\varepsilon}(u_n),u_n\rangle\\
&=(\frac{1}{2}-\frac{1}{\alpha})\|u_n\|_{\varepsilon}^2-\frac{1}{2}\int_{\R^N}(\frac{1}{|x|^{\mu}}\ast F(u_n^+))(F(u_n^+)-\frac{2}{\alpha}f(u_n^+)u_n^+)dx+(\frac{1}{\alpha}-\frac{1}{2^*_s})\int_{\R^N}|u_n^+|^{2^*_s}dx\\
&\geq(\frac{1}{2}-\frac{1}{\alpha})\|u_n\|_{\varepsilon}^2.
\end{split}
\end{equation*}
Therefore, we get that the sequence $\{u_n\}$ is bounded in $H_{\varepsilon}.$
Next, we prove that ${\|u_n^-\|}=o_n(1).$ Since $\langle J'_{\varepsilon}(u_n),u_n^-\rangle=o_n(1),$ by using $f(t)=0$ for $t\leq0$ and $(x-y)(x^--y^-)\geq|x^--y^-|^2$ where $x^-=\min\{x,0\},$ we can deduce that
\begin{equation*}
\begin{split}
\|u_n^-\|^2_{\varepsilon}&\leq\int_{\R^{N}}\int_{\R^{N}}\frac{(u_n(x)-u_n(y))(u_n^-(x)-u_n^-(y))}{|x-y|^{N+2s}}dxdy+\int_{\R^N}V(\varepsilon x)u_nu_n^-dx\\
&=\int_{\R^N}(\frac{1}{|x|^{\mu}}\ast F(u_n^+))f(u_n^+)u_n^-dx+\int_{\R^N}|u_n^+|^{2^*_s-2}u_n^+u_n^-dx+o_n(1)\\
&=o_n(1).
\end{split}
\end{equation*}
Therefore, we complete our proof.
\end{proof}
\begin{lemma}\label{Lem3.4}
There exists a constant $r>0$ such that $\|u\|_{\varepsilon}\geq r$ for all $\varepsilon\geq0$ and $u\in\mathcal{N}_{\varepsilon}$
\end{lemma}
\begin{proof}
By using Lemma \ref{Lem2.3} and $(f_1)-(f_2),$ we can see that for any $u\in\mathcal{N}_{\varepsilon}$
$$\|u\|^2_{\varepsilon}\leq C(\|u\|^4_{\varepsilon}+\|u\|^{2q}_{\varepsilon}+\|u\|^{2^*_s}_{\varepsilon})$$
then, there exists $r>0$ such that
\begin{equation}\label{eq3.5}
\|u\|_{\varepsilon}\geq r \ \ \text{for all} \ \ \ u\in\mathcal{N}_{\varepsilon}, \,\,\ \varepsilon\geq0
\end{equation}
Hence, we deduce to the lemma holds.
\end{proof}
\par
When $V\equiv1,$ then $H^s(\R^N)=H_{\varepsilon}(\R^N).$ For $\tau>0$ and $u\in H^s(\R^N),$ let
$$I_{\tau}(u)=\frac{1}{2}\int_{\R^N}\int_{\R^N}\frac{|u(x)-u(y)|^2}{|x-y|^{N+2s}}dxdy+\frac{\tau}{2}\int_{\R^N}u^2dx-\frac{1}{2^*_s}\int_{\R^N}|u^+|^{2^*_s}dx-\frac{1}{2}\int_{\R^N}(\frac{1}{|x|^{\mu}}\ast F(u^+))F(u^+)dx$$
$$\M_{\tau}:=\{u\in H^s(\R^N):u^+\neq0,\langle I'_{\tau}(u),u\rangle=0\} ,\,\,\ m_{\tau}:=\inf\limits_{\M_{\tau}}I_{\tau}.$$
\par
For $m_{\tau},$ there also holds
$$m_{\tau}=\inf\limits_{\gamma\in\Gamma}\sup\limits_{t\in[0,1]}I_{\tau}(\gamma(t))=\inf\limits_{u\in H^s(\R^N)}\sup\limits_{t\geq0}I_{\tau}(tu)$$
where $\Gamma=\{\gamma\in C([0,1],H^s(\R^N)):\gamma(0)=0, \,\,\ I_{\tau}(\gamma(1))<0\}.$
\begin{lemma}\label{Lem3.5}
For any $\tau>0,$ there exists $u\in H^s(\R^N)$ with $u^+\neq0$ such that $$\max\limits_{t\geq0}I_{\tau}(tu)<\frac{s}{N}S^{\frac{N}{2s}},$$
where $S:=\inf\limits_{u\in\mathcal{D}^{s,2}(\R^N)}\frac{\|u\|^2_{\mathcal{D}^{s,2}(\R^N)}}{\|u\|_{L^{2^*_s}(\R^N)}^2}.$
\end{lemma}
\begin{proof}
Let $\varphi\in C^{\infty}_0(\R^N)$ be such that $\varphi=1$ in $B_{\delta},$ $\varphi(x)=0$ in $\R^N\setminus B_{2\delta}.$ Denote
$$U_{\varepsilon}(x)=\varepsilon^{-\frac{N-2s}{2}}u^*(\frac{x}{\varepsilon})$$
where $u^*(x)=\frac{\tilde{u}(x/{S^{\frac{1}{2s}}})}{\|\tilde{u}\|_{L^{2^*_s}(\R^N)}},$ $\tilde{u}(x/{S^{\frac{1}{2s}}})=\frac{\alpha}{(1+|x/{S^{\frac{1}{2s}}}|^2)^{\frac{N-2s}{2}}}$ with $\alpha>0.$ We define $$u_{\varepsilon}(x):=\varphi(x)U_{\varepsilon}(x)$$
then $u_{\varepsilon}\in H_{\varepsilon}.$ From \cite{Nezza2012BDSM} and \cite{Sreenadh17NODEA}, we have the following estimations
\begin{equation}\label{eq3.6}
[u_{\varepsilon}]_{H^s(\R^N)}\leq S^{\frac{N}{2s}}+O(\varepsilon^{N-2s})
\end{equation}

\begin{equation}\label{eq3.7}
\int_{\R^N}|u_{\varepsilon}(x)|^{2^*_s}dx=S^{\frac{N}{2s}}+O(\varepsilon^N)
\end{equation}
\begin{equation}\label{eq3.8}
\int_{\R^N}|u_{\varepsilon}|^2dx=\left\{
\begin{array}{ll}  C_s\varepsilon^{2s}+O(\varepsilon^{N-2s}), \,\,\ \,\,\ \,\,\ \,\,\ \,\,\ \,\,\ \,\,\ \,\,\ if \,\,\ N>4s,\\
C_s\varepsilon^{2s}|ln\varepsilon|+O(\varepsilon^{2s}), \,\,\ \,\,\ \,\,\ \,\,\ \,\,\ \,\,\ \,\,\ if \,\,\ N=4s,\\
C_s\varepsilon^{N-2s}+O(\varepsilon^{N-2s}), \,\,\ \,\,\ \,\,\ \,\,\ \,\,\ \,\,\ if \,\,\ N<4s.
\end{array} \right.
\end{equation}
A standard argument shows that for any $u_{\varepsilon},$ there exists a unique $t_{\varepsilon}$ such that $t_{\varepsilon}u_{\varepsilon}\in\M_{\tau}$ and $I_{\tau}(t_{\varepsilon}u_{\varepsilon})=\max\limits_{t\geq0}I_{\tau}(tu_{\varepsilon}).$ As a consequence $m_{\tau}\leq I_{\tau}(t_{\varepsilon}u_{\varepsilon})$ and
$$[u_{\varepsilon}]_{H^s(\R^N)}^2+\tau\int_{\R^N}u_{\varepsilon}^2dx=t^{-1}_{\varepsilon}\int_{\R^N}(\frac{1}{|x|^{\mu}}\ast F(t_{\varepsilon}u_{\varepsilon}))f(t_{\varepsilon}u_{\varepsilon})u_{\varepsilon}dx+t^{2^*_s-2}_{\varepsilon}\int_{\R^N}|u_{\varepsilon}|^{2^*_s}dx.$$
As a consequence $t_{\varepsilon}\geq t_0,$ where $t_0>0$ is independent of $\varepsilon.$
Now, we estimate the convolution term. For $\varepsilon>0$ small enough, we have
\begin{equation}\label{eq3.9}
\begin{split}
\int_{\R^N}(\frac{1}{|x|^{\mu}}\ast F(t_\varepsilon u_{\varepsilon}))F(t_\varepsilon u_{\varepsilon})dx&\geq Ct_\varepsilon^{2\sigma}\int_{\R^N}(\frac{1}{|x|^{\mu}}\ast |u_{\varepsilon}|^{\sigma})|u_{\varepsilon}|^{\sigma}dx\\
&\geq Ct_0^{2\sigma}\int_{B_{\delta}}\int_{B_{\delta}}\frac{|u_{\varepsilon}(y)|^{\sigma}|u_{\varepsilon}(x)|^{\sigma}}{|x-y|^{\mu}}dxdy\\
&\geq Ct_0^{2\sigma}\int_{B_{\delta}}\int_{B_{\delta}}\frac{C_1\varepsilon^{\sigma(N-2s)}}{(\varepsilon^2+|y|^2)^{\frac{\sigma(N-2s)}{2}}(\varepsilon^2+|x|^2)^{\frac{\sigma(N-2s)}{2}}}dxdy\\
&\geq Ct_0^{2\sigma}\int_{B_{\frac{\delta}{2}}}\int_{B_{\frac{\delta}{2}}}\frac{C_2\varepsilon^{2N-\sigma(N-2s)}}{(1+|x|^2)^{\frac{\sigma(N-2s)}{2}}(1+|y|^2)^{\frac{\sigma(N-2s)}{2}}}dxdy\\
&\geq O(\varepsilon^{2N-\sigma(N-2s)})
\end{split}
\end{equation}
Set $g(t):=\frac{t^2}{2}([u_{\varepsilon}]_{H^s(\R^N)}^2+\tau\int_{\R^N}u_\varepsilon^2dx)-\frac{t^{2^*_s}}{2^*_s}\int_{\R^N}|u_{\varepsilon}|^{2^*_s}dx.$ If $N>4s,$
by a simple calculation, we get
\begin{equation}\label{eq3.10}
\begin{split}
\max\limits_{t\geq0}g(t)&=\frac{s}{N}(\frac{[u_{\varepsilon}]_{H^s{\R^N}}^2+\tau\int_{\R^N}u_\varepsilon^2dx}{\|u_{\varepsilon}\|^2_{2^*_s}})^{\frac{N}{2s}}\\
&=\frac{s}{N}(\frac{S^{\frac{N}{2s}}+O(\varepsilon^{N-2s})+O(\varepsilon^{2s})}{(S^{\frac{N}{2s}}+O(\varepsilon^N))^{\frac{N-2s}{N}}})^{\frac{N}{2s}}\\
&=\frac{s}{N}S^{\frac{N}{2s}}+O(\varepsilon^{N-2s})+O(\varepsilon^{2s})
\end{split}
\end{equation}
Nothing that $\sigma\in[q_{_N},\frac{2N-\mu}{N-2s}),$ for $\varepsilon>0$ small enough, using \eqref{eq3.9},\eqref{eq3.10} we can check
\begin{equation*}
\begin{split}
\max\limits_{t\geq0}I_{\tau}(tu_{\varepsilon})&\leq\max\limits_{t\geq0}g(t)-\int_{\R^N}(\frac{1}{|x|^{\mu}}\ast F(t_\varepsilon u_{\varepsilon}))F(t_\varepsilon u_{\varepsilon})dx\\
&<\frac{s}{N}S^{\frac{N}{2s}}+O(\varepsilon^{N-2s})-O(\varepsilon^{2N-\sigma(N-2s)})+O(\varepsilon^{2s})\\
&<\frac{s}{N}S^{\frac{N}{2s}}.
\end{split}
\end{equation*}
In similar way, we can check $N=4s$ and $N<4s.$
\end{proof}
\begin{lemma}\label{Lem3.6}
Let $\{u_n\}\subset H_{\varepsilon}$ be a $(PS)_d$ sequence of $J_{\varepsilon}$ with $d<\frac{s}{N}S^\frac{N}{2s}$ and $u_n\rightharpoonup0$ in $H_{\varepsilon}.$ Then one of the following conclusions holds:
\begin{itemize}
\item[$(a)$] $u_n\rightarrow0$ in $H_{\varepsilon};$
\item[$(b)$] There exists a sequence $\{y_n\}\subset\R^N$ and positive constants $r,\beta,$ such that $$\liminf\limits_{n\rightarrow\infty}\int_{B_r(y_n)}|u_n(x)|^2dx>\beta.$$
\end{itemize}
\end{lemma}
\begin{proof}
If $(b)$ does not occur, then for all $R>0,$ up to a subsequence
$$\lim\limits_{n\rightarrow\infty}\sup\limits_{y\in\R^N}\int_{B_R(y)}|u_n(x)|^2dx=0.$$
Since we know that $\{u_n\}$ is bounded in $H_{\varepsilon},$ we can use Lemma \ref{Lem2.2} to deduce that $u_n\rightarrow0$ in $L^r(\R^N)$ for any $r\in(2,2^*_s).$
So, apply Hardy-Littlewood-Sobolev inequality, we know that
\begin{equation}\label{eq3.11}
\int_{\R^N}(\frac{1}{|x|^{\mu}}\ast F(u^+_n))F(u^+_n)dx=o_n(1).
\end{equation}
Taking into account $\langle J'_{\varepsilon}(u_n),u_n\rangle=o_n(1)$ we can infer that $$\|u_n\|^2_{\varepsilon}=\|u_n\|^{2^*_s}_{L^{2^*_s}(\R^N)}+o_n(1).$$
Since $\{u_n\}$ is bounded, up to a subsequence, we have
$$\|u_n\|^2_{\varepsilon}\rightarrow l\geq0 \,\,\ \text{and} \,\,\ \|u_n\|^{2^*_s}_{L^{2^*_s}(\R^N)}\rightarrow l\geq0.$$
If $l>0,$ then
\begin{equation*}
\begin{split}
S&\leq\frac{\int_{\R^N}|(-\Delta)^{\frac{s}{2}}u_n|^2dx}{(\|u_n\|^{2^*_s}_{L^{2^*_s}(\R^N)})^{\frac{2}{2^*_s}}}\\
&\leq\frac{\|u_n\|^2_{\varepsilon}}{(\|u_n\|^{2^*_s}_{L^{2^*_s}(\R^N)})^{\frac{2}{2^*_s}}}\\
&\rightarrow l^{\frac{2s}{N}}
\end{split}
\end{equation*}
as $n\rightarrow \infty,$ hence $l\geq S^{\frac{N}{2s}}.$ Consequently, by \eqref{eq3.11}, we have
\begin{equation*}
\begin{split}
d&=\lim\limits_{n\rightarrow\infty}J_{\varepsilon}(u_n)\\
&=\lim\limits_{n\rightarrow\infty}\bigg(\frac{1}{2}\|u_n\|^2_{\varepsilon}-\frac{1}{2}\int_{\R^N}(\frac{1}{|x|^{\mu}}\ast F(u^+_n))F(u^+_n)dx-\frac{1}{2^*_s}\int_{\R^N}|u^+_n|^{2^*_s}dx\bigg)\\
&=\frac{s}{N}l\\
&>\frac{s}{N}S^{\frac{N}{2s}}
\end{split}
\end{equation*}
a contradiction, hence $l=0.$ Consequently, by the boundedness of $\{u_n\}$ in $H_{\varepsilon},$ we have $u_n\rightarrow0$ in $H_{\varepsilon},$ so $(a)$ holds. This completes the proof.
\end{proof}
\begin{lemma}\label{lem3.7}
Let $\{u_n\}\subset H_{\varepsilon}$ be a $(PS)_d$ sequence of $J_{\varepsilon}$ with $d<m_{_{V_{\infty}}}$ and $u_n\rightharpoonup0$ in $H_{\varepsilon}.$ Then $u_n\rightarrow0$ in $H_{\varepsilon}.$
\end{lemma}
\begin{proof}
By Lemma \ref{Lem3.3} we can assume $u_n\geq0.$ For any subsequence of $\{u_n\}$ still denoted by $\{u_n\}.$ Since $u_n\rightharpoonup0$ in $H_{\varepsilon},$ up to a subsequence, we can assume $$u_n\rightarrow0 \text{ in } L^r_{loc}(\R^N) \,\,\ r\in[2,2^*_s) \,\,\ \text{ and } \,\,\ u_n(x)\rightarrow0 \,\,\ a.e. \,\,\ x\in\R^N.$$ If $u_n\nrightarrow0$ in $H_{\varepsilon},$ by Lemma \ref{Lem3.2} we know for any $\{t_n\}\subset(0,+\infty)$ $s.t.$ $\{t_nu_n\}\subset\mathcal{N}_{V_{\infty}}.$\\
Claim $\limsup\limits_{n\rightarrow\infty}\leq1.$
If does not occur for any $\delta>0,$ consider any subsequence of $\{t_n\}$ and satisfies the following
$$t_n\geq1+\delta, \,\,\ \,\,\ \,\,\ \forall \,\,\ n\in\N.$$
Since $\{u_n\}$ is a $(PS)_d$ sequence of $J_{\varepsilon},$ we can see that
\begin{equation}\label{eq3.12}
[u_n]^2_{H^s(\R^N)}+\int_{\R^N}V(\varepsilon x)|u_n|^2dx=\int_{H^s(\R^N)}(\frac{1}{|x|^{\mu}}\ast F(u_n))f(u_n)u_ndx+\int_{\R^N}|u_n|^{2^*_s}dx+o_n(1).
\end{equation}
We observe that $\{t_nu_n\}\subset\mathcal{N}_{V_{\infty}},$ we have
\begin{equation}\label{eq3.13}
t^2_n[u_n]^2_{H^s(\R^N)}+t^2_n\int{\R^N}V_{\infty}u_n^2dx=\int_{\R^N}(\frac{1}{|x|^{\mu}}\ast F(t_nu_n))f(t_nu_n)t_nu_ndx+\int_{\R^N}|t_nu_n|^{2^*_s}dx.
\end{equation}
Taking into account \eqref{eq3.12} and \eqref{eq3.13} we can deduce that
\begin{equation*}
\begin{split}
\int_{\R^N}(V_{\infty}-V(\varepsilon x))|u_n|^2dx=&\int_{\R^N}\big(\frac{(\frac{1}{|x|^{\mu}}\ast F(t_nu_n))f(t_nu_n)u_n^2}{t_nu_n}-\frac{(\frac{1}{|x|^{\mu}}\ast F(u_n))f(u_n)u_n^2}{u_n}\big)dx\\
&+\int_{\R^N}\big(\frac{|t_nu_n|^{2^*_s}}{t_n^2}-|u_n|^{2^*_s}\big)dx+o_n(1).
\end{split}
\end{equation*}
By $(V_0)$, for any $\xi>0$ there exists $R(\xi):=R>0$ such that $$V(\varepsilon x)\geq V_{\infty}-\xi, \,\,\ \,\,\,\ \,\,\ |\varepsilon x|\geq R.$$
Notice that $u_n\rightarrow0$ in $L^2(B_R(0))$ and the boundedness of $\{u_n\}$ in $H_{\varepsilon},$ we get
\begin{equation*}
\begin{split}
&\int_{\R^N}\big(\frac{(\frac{1}{|x|^{\mu}}\ast F(t_nu_n))f(t_nu_n)u_n^2}{t_nu_n}-\frac{(\frac{1}{|x|^{\mu}}\ast F(u_n))f(u_n)u_n^2}{u_n}\big)dx\\
\leq&\int_{\R^N}(V_{\infty}-V(\varepsilon x))|u_n|^2dx\\
=&\int_{B_R(0)}(V_{\infty}-V(\varepsilon x))|u_n|^2dx+\int_{B^c_R(0)}(V_{\infty}-V(\varepsilon x))|u_n|^2dx\\
\leq&V_{\infty}\int_{B_R(0)}|u_n|^2dx+\xi\int_{B_R^c(0)}|u_n|^2dx\\
\leq&o_n(1)+\frac{\xi}{V_0}\int_{B_R^c(0)}V(\varepsilon x)|u_n|^2dx\\
\leq&o_n(1)+\frac{\xi}{V_0}\|u_n\|^2_{\varepsilon}\leq o_n(1)+\xi C.
\end{split}
\end{equation*}
If $u_n\nrightarrow0,$ $\exists$ $\{y_n\}\subset\R^N,$ $r,\delta>0$ such that
$$\liminf\limits_{n\rightarrow\infty}\int_{B_r(y_n)}|u_n(x)|^2dx\geq\delta.$$
Let $\tilde{u}_n(x)=u_n(x+y_n)$ then there exists $\tilde{u},$ up to a subsequence, we have
$$\tilde{u}_n\rightharpoonup\tilde{u} \text{ in } H^s(\R^N), \,\,\ \,\,\ \tilde{u}_n\rightarrow\tilde{u} \text{ in } L_{loc}^t(\R^N), \,\,\ \,\,\ \tilde{u}_n(x)\rightarrow\tilde{u}(x) \,\,\ a.e. \,\,\ x\in\R^N.$$
So, $\exists\Omega\subset B_r(0)$ $s.t.$ $\tilde{u}>0$ in $\Omega,$ we can infer
$$\int_{\Omega}\big(\frac{(\frac{1}{|x|^{\mu}}\ast F((1+\delta)\tilde{u}))f((1+\delta)\tilde{u})}{(1+\delta)\tilde{u}}-(\frac{\frac{1}{|x|^{\mu}}\ast F(\tilde{u}))f(\tilde{u})}{\tilde{u}}\big)\tilde{u}^2dx\leq\xi C+o_n(1)$$
Taking the limit as $n\rightarrow\infty$ and by applying Fatou's lemma we obtain
$$0<\int_{\Omega}\int_{\Omega}\big(\frac{F((1+\delta)\tilde{u}(y))}{|x-y|^{\mu}}\frac{f((1+\delta)\tilde{u}(x))}{(1+\delta)\tilde{u}(x)}-\frac{F(\tilde{u}(y))}{|x-y|^{\mu}}\frac{f(\tilde{u}(x))}{\tilde{u}(x)}\big)\tilde{u}^2dxdy\leq\xi C$$
For any $\xi>0$, this gives a contradiction. Therefore, $\limsup\limits_{n\rightarrow\infty}t_n\leq1.$
Case1:Assume that $\limsup\limits_{n\rightarrow\infty}t_n=1.$ Hence there exists a subsequence of $\{t_n\},$ still denoted by $\{t_n\}$ such that $t_n\rightarrow1.$ Clearly,
$$d+o_n(1)=J_{\varepsilon}(u_n)\geq J_{\varepsilon}(u_n)+m_{_{V_{\infty}}}-I_{V_{\infty}}(t_nu_n).$$
Moreover,
\begin{equation*}
\begin{split}
J_{\varepsilon}(u_n)-I_{V_{\infty}}(t_nu_n)\geq&\frac{1-t_n^2}{2}[u_n]^2_{H^s(\R^N)}+\frac{1}{2}\int_{\R^N}(V(\varepsilon x)-t^2_nV_{\infty})|u_n|^2dx+\frac{1}{2^*_s}\int_{\R^N}(|t_nu_n|^{2^*_s}-|u_n|^{2^*_s})dx\\
&+\frac{1}{2}\int_{\R^N}\big((\frac{1}{|x|^{\mu}}\ast F(t_nu_n))F(t_nu_n)-(\frac{1}{|x|^{\mu}}\ast F(u_n))F(u_n)\big)dx.
\end{split}
\end{equation*}
Since $\{u_n\}$ is bounded in $H_{\varepsilon},$ by using the Mean Value Theorem and $t_n\rightarrow1,$ we have
\begin{equation*}
\begin{split}
\int_{\R^N}(V(\varepsilon x)-t_n^2V_{\infty})|u_n|^2dx&=\int_{B_R(0)}(V(\varepsilon x)-t_n^2V_{\infty})|u_n|^2dx+\int_{B_R^c(0)}(V(\varepsilon x)-t_n^2V_{\infty})|u_n|^2dx\\
&\geq(V_0-t_n^2V_{\infty})\int_{B_R(0)}|u_n|^2dx-\xi\int_{B^c_R(0)}|u_n|^2dx+V_{\infty}(1-t_n^2)\int_{B^c_R(0)}|u_n|^2dx\\
&\geq o_n(1)-\xi C
\end{split}
\end{equation*}
For any $\xi$ and this gives a contradiction.\\
Case2: $\limsup\limits_{n\rightarrow\infty}t_n:=t_0<1.$ Then there exists a subsequence of $\{t_n\},$ still denoted by $\{t_n\}$ such that $t_n\rightarrow t_0$ and $t_n<1$ for any $n\in\N,$ we deduce that
\begin{equation*}
\begin{split}
m_{_{V_{\infty}}}&\leq I_{V_{\infty}}(t_nu_n)\\
&=J_{\varepsilon}(t_nu_n)+\frac{t_n^2}{2}\int_{\R^N}(V_{\infty}-V(\varepsilon x))|u_n|^2dx\\
&=J_{\varepsilon}(t_nu_n)+C\xi+o_n(1)\\
&=d+C\xi+o_n(1).
\end{split}
\end{equation*}
This gives a contradiction.
\end{proof}
\par
By similar argument as the Lemma3.1 in \cite{Alves2014Multiplicity} and Lemma4.7 in \cite{Zhang2015Solutions}, we have the following lemma.
\begin{lemma}\label{lem3.8}
Let $\{u_n\}$ be a sequence such that $u_n\rightharpoonup u$ in $H_{\varepsilon}$ and $w_n:=u_n-u.$ Then, we have
\begin{itemize}
\item[$(i)$]  $\int_{\R^N}|F(w_n)-F(u_n)+F(u)|^tdx=o_n(1)$ where $t=\frac{2N}{2N-\mu}.$
\item[$(ii)$]  $\int_{\R^N}(\frac{1}{|x|^{\mu}}\ast F(u_n-u))F(u_n-u)dx-\int_{\R^N}(\frac{1}{|x|^{\mu}}\ast F(u_n))F(u_n)dx+\int_{\R^N}(\frac{1}{|x|^{\mu}}\ast F(u))F(u)dx=o_n(1).$
\item[$(iii)$]  $\forall$ $\xi>0,$ we have $$\int_{\R^N}|f(u_n-u)-f(u_n)+f(u)|^t|\varphi|^tdx\leq C\xi\|\varphi\|^t_{\varepsilon}\leq C\xi. \,\,\ \,\,\ \forall \varphi\in H_{\varepsilon}(\R^N),\,\,\ \|\varphi\|_{\varepsilon}=1.$$
\item[$(iv)$]  $$\big|\int_{\R^N}(\frac{1}{|x|^{\mu}}\ast F(u_n-u))f(u_n-u)\varphi dx-\int_{\R^N}(\frac{1}{|x|^{\mu}}\ast F(u_n))f(u_n)\varphi dx+\int_{\R^N}(\frac{1}{|x|^{\mu}}\ast F(u))f(u)\varphi dx|\leq C\xi\|\varphi\|_{\varepsilon}.$$
    where, $\xi>0,$ $\varphi\in H_{\varepsilon}(\R^N).$
\end{itemize}
\end{lemma}
\begin{proof}
$(i)$ By the Mean Value Theorem and \eqref{eq3.1}, it follows that
\begin{equation*}
\begin{split}
|F(w_n)-F(u_n)|&=|\int^1_0(\frac{d}{dt}F(u_n-tu))dt|\\
&\leq\int^1_0|uf(u_n-tu)|dt\\
&\leq\int^1_0(\xi|u||u_n-tu|+C_{\xi}|u||u_n-tu|^{q-1})dt\\
&\leq\xi|u_n||u|+\xi|u|^2+C_{\xi}|u_n|^{q-1}|u|+C_{\xi}|u|^{q}.
\end{split}
\end{equation*}
By applying Young inequality with $\delta>0,$ we get
$$|F(w_n)-F(u_n)|\leq\delta(|u_n|^2+|u_n|^{q})+C_{\delta}(|u|^2+|u|^{q})$$
which yields
$$|F(w_n)-F(u_n)+F(u)|\leq\delta(|u_n|^2+|u_n|^{q})+C_{\delta}(|u|^2+|u|^{q})+C(|u|^2+|u|^{q}).$$
\begin{equation*}
\begin{split}
|F(w_n)-F(u_n)+F(u)|^t&\leq4^t\delta(|u_n|^{2t}+|u_n|^{qt})+C(|u_n|^{2t}+|u_n|^{qt})\\
&\leq4^t\delta(|u_n|^{2t}+|u_n|^{qt}-|u_n|^{2t}+|u_n|^{qt})+C_1(|u_n|^{2t}+|u_n|^{qt})
\end{split}
\end{equation*}
Let $$G_{\delta,n}(x)=\max\big\{|F(w_n)-F(u_n)+F(u)|^t-4^t\delta(|u_n|^{2t}+|u_n|^{qt}-|u|^{2t}-|u|^{qt}),0\big\}.$$
Then $G_{\delta,n}\rightarrow0$ $a.e.$ in $\R^N$ as $n\rightarrow\infty$ and $0\leq G_{\delta,n}\leq C_1(|u|^{2t}+|u|^{qt})\in L^1(\R^N).$ As a consequence of the Dominated Convergence Theorem, we have $$\int_{\R^N}G_{\delta,n}(x)dx\rightarrow0 \,\,\ as \,\,\ n\rightarrow\infty.$$
On the other hand, from the definition of $G_{\delta,n}$, we get $$|F(w_n)-F(u_n)+F(u)|^t\leq4^t\delta(|u_n|^{2t}+|u_n|^{qt})+G_{\delta,n}$$
which together with the boundedness of $\{u_n\}$ gives
$$\limsup\limits_{n\rightarrow\infty}\int_{\R^N}|F(w_n)-F(u_n)+F(u)|^tdx\leq C\delta \,\,\ \text{for some} \,\,\ C>0.$$
As $\delta$ is arbitrary, we obtain
$$\int_{\R^N}|F(w_n)-F(u_n)+F(u)|^tdx=o_n(1).$$
$(ii)$
\begin{equation*}
\begin{split}
&\int_{\R^N}(\frac{1}{|x|^{\mu}}\ast F(u_n-u))F(u_n-u)dx-\int_{\R^N}(\frac{1}{|x|^{\mu}}\ast F(u_n))F(u_n)dx+\int_{\R^N}(\frac{1}{|x|^{\mu}}\ast F(u))F(u)dx\\
=&\int_{\R^N}(\frac{1}{|x|^{\mu}}\ast F(u_n-u))(F(u_n-u)-F(u_n)+F(u))dx+\int_{\R^N}(\frac{1}{|x|^{\mu}}\ast F(u_n))(F(u_n-u)-F(u_n)+F(u))dx\\
& \,\,\ +\int_{\R^N}(\frac{1}{|x|^{\mu}}\ast F(u))(F(u_n-u)-F(u_n)+F(u))dx-2\int_{\R^N}(\frac{1}{|x|^{\mu}}\ast F(u))F(u_n-u)dx\\
=&:I_1+I_2+I_3+I_4
\end{split}
\end{equation*}
By the boundedness of $\{u_n\}$ and $(f_1)-(f_2).$ we know that $$\big(\int_{\R^N}|F(u_n-u)|^tdx\big)^{\frac{1}{t}}\leq C.$$
From Lemma \ref{Lem2.3}, we have
$$|I_1|\leq\big(\int_{\R^N}|F(u_n-u)|^tdx\big)^{\frac{1}{t}}\big(\int_{\R^N}|F(u_n-u)-F(u_n)-F(u)|^tdx\big)^{\frac{1}{t}}\rightarrow0$$
Likewise, $I_2\rightarrow0,$ $I_3\rightarrow0.$
By the boundedness of $\{u_n\},$ we have $\{F(u_n-u)\}$ is bounded in $L^{\frac{2N}{2N-\mu}}(\R^N)$ and $F(u_n-u)\rightarrow0$ $a.e.$ in $\R^N.$
So, $F(u_n-u)\rightharpoonup0$ in $L^{\frac{2N}{2N-\mu}}(\R^N).$ In view of $\frac{1}{|x|^{\mu}}\ast F(u)\in\big(L^{\frac{2N}{2N-\mu}}(\R^N)\big)^{\ast},$ we obtain $$I_4=-2\int_{\R^N}(\frac{1}{|x|^{\mu}}\ast F(u))F(u_n-u)dx\rightarrow0 \text{ as } n\rightarrow\infty$$
Therefore, we can conclude $(ii)$ holds.\\
$(iii)$ By using $(f_1)$ and $(f_2),$ we know that for any $\xi>0,$ there exists $N_0\in(0,1)$ and $N_1>2$ such that  $$|f(t)|\leq\xi|t| \,\,\ \,\,\ \text{ in } \,\,\ \,\,\ |t|\leq2N_0,$$ $$|f(t)|\leq\xi|t|^{q-1} \,\,\ \,\,\ \text{ in } \,\,\ \,\,\ |t|\geq N_1-1,$$ $$|f(t)|\leq C_{\xi}|t|+\xi|t|^{q-1} \,\,\ \text{ for } \,\,\ t\in\R.$$ Since $f$ is a continuous function, we deduce that exists $\delta\in(0,N_0)$ such that
$$|f(t_1)-f(t_2)|\leq N_0\xi, \,\,\ \,\,\ \forall |t_1|\leq N_0+N_1,|t_2|\leq N_0+N_1 \text{ and }|t_1-t_2|\leq\delta.$$
Taking into account $u\in H_{\varepsilon}(\R^N),$ we know that there exists $R_0>0$ such that
$$\big(\int_{B^c_{R_0}(0)}|u|^{2t}dx\big)^{\frac{1}{2}}<\xi, \,\,\ \big(\int_{B^c_{R_0}(0)}|u|^{tq}dx\big)^{\frac{q-1}{q}}<\xi$$
For any $\varphi\in H_{\varepsilon}(\R^N),$ $\|\varphi\|_\varepsilon=1,$ we have
\begin{equation*}
\begin{split}
\int_{B_{R_0}^c(0)}|f(u)\varphi|^{\frac{2N}{2N-\mu}}dx&\leq\int_{B_{R_0}^c(0)}(\xi|u||\varphi|+C_{\xi}|u|^{q-1})^{\frac{2N}{2N-\mu}}|\varphi|^{\frac{2N}{2N-\mu}}dx\\
&\leq\int_{B^c_{R_0}(0)}(2^t\xi|u|^t+C|u|^{t(q-1)})|\varphi|^tdx\\
&\leq 2^t\xi\big(\int_{B^c_{R_0}(0)}|u|^2dx\big)^{\frac{1}{2}}\big(\int_{B^c_{R_0}(0)}|\varphi|^{2t}dx\big)^{\frac{1}{2}}+C\big(\int_{B^c_{R_0}(0)}|u|^{tq}dx\big)^{\frac{q-1}{q}}\big(\int_{B^c_{R_0}(0)}|\varphi|^{qt}dx\big)^{\frac{1}{q}}\\
&\leq C\xi\|\varphi\|_{\varepsilon}^t
\end{split}
\end{equation*}
We denote $A_n:=\big\{B^c_{R_0}(0):|u_n(x)|\leq N_0\big\},$ $B_n:=\big\{B^c_{R_0}(0):|u_n(x)|\geq N_1\big\},$ $C_n:=\big\{B^c_{R_0}(0):N_0<|u_n(x)|< N_1\big\}$
\begin{equation}\label{eq3.14}
\begin{split}
&\int_{A_n\bigcap\{|u|\leq\delta\}}|f(u_n-u)-f(u_n)|^t|\varphi|^tdx\\
\leq&\xi\int_{A_n\bigcap\{|u|\leq\delta\}}(|u_n-u|+|u_n|)^t|\varphi|^t dx\\
\leq&2^t\xi^t(\int_{A_n\bigcap\{|u|\leq\delta\}}|u_n-u|^t|\varphi|^tdx+\int_{A_n\bigcap\{|u|\leq\delta\}}|u_n|^t|\varphi|^tdx)\\
\leq&2^t\xi^t\big( (\int_{A_n\bigcap\{|u|\leq\delta\}}|u_n-u|^{2t}dx)^{\frac{1}{2}}+(\int_{A_n\bigcap\{|u|\leq\delta\}}|u_n|^{2t}dx)^{\frac{1}{2}}\big)(\int_{A_n\bigcap\{|u|\leq\delta\}}|\varphi|^{2t}dx)^{\frac{1}{2}}\\
\leq&\xi^t C\|\varphi\|_{\varepsilon}^t
\end{split}
\end{equation}
\begin{equation}\label{eq3.15}
\begin{split}
&\int_{B_n\bigcap\{|u|\leq\delta\}}|f(u_n-u)-f(u_n)|^t|\varphi|^tdx\\
\leq&\xi^t(\int_{B_n\bigcap\{|u|\leq\delta\}}(|u_n-u|^{q-1}+|u_n|^{q-1})^t|\varphi|^tdx\\
\leq&2^t\xi^t\big(\int_{B_n\bigcap\{|u|\leq\delta\}}|u_n-u|^{t(q-1)}|\varphi|^tdx+\int_{B_n\bigcap\{|u|\leq\delta\}}|u_n|^{t(q-1)}|\varphi|^tdx\big)\\
\leq&2^t\xi^t\big[(\int_{B_n\bigcap\{|u|\leq\delta\}}|u_n-u|^{tq}dx)^{\frac{q-1}{q}}(\int_{B_n\bigcap\{|u|\leq\delta\}}|\varphi|^{tq}dx)^{\frac{1}{q}}+(\int_{B_n\bigcap\{|u|\leq\delta\}}|u_n-u|^{tq}dx)^{\frac{q-1}{q}}(\int_{B_n\bigcap\{|u|\leq\delta\}}|\varphi|^{tq}dx)^{\frac{1}{q}}\big]\\
\leq&2^t\xi^t\big[(\int_{B_n\bigcap\{|u|\leq\delta\}}|u_n-u|^{tq}dx)^{\frac{q-1}{q}}+(\int_{B_n\bigcap\{|u|\leq\delta\}}|u_n-u|^{tq}dx)^{\frac{q-1}{q}}\big](\int_{B_n\bigcap\{|u|\leq\delta\}}|\varphi|^{tq}dx)^{\frac{1}{q}}\\
\leq&\xi^tC\|\varphi\|_{\varepsilon}^t
\end{split}
\end{equation}
\begin{equation}\label{eq3.16}
\int_{C_n\bigcap\{|u|\leq\delta\}}|f(u_n-u)-f(u_n)|^t|\varphi|^tdx\leq N_0^t\xi^t\int_{C_n}|\varphi|^tdx\leq N_0^t\xi^t|C_n|^{\frac{1}{2}}(\int_{\R^N}|\varphi|^{2t}dx)^{\frac{1}{2}}\leq\xi^t C\|\varphi\|_{\varepsilon}
\end{equation}
Thus, putting together \eqref{eq3.14}, \eqref{eq3.15} and \eqref{eq3.16} we get
$$\int_{(B^c_{R_0}(0))\bigcap\{|u|\leq\delta\}}|f(u_n)-f(u_n-u)|^t|\varphi|^tdx\leq C\xi\|\varphi\|_{\varepsilon}.$$
Moreover,
\begin{equation*}
\begin{split}
&\int_{(B^c_{R_0}(0))\bigcap\{|u|>\delta\}}|f(u_n)-f(u_n-u)|^t|\varphi|^tdx\\
\leq&C_{\xi}\int_{(B^c_{R_0}(0))\bigcap\{|u|>\delta\}}2^t(|u_n-u|^{t}|\varphi|^t+|u_n|^{t}|\varphi|^t)dx+\xi\int_{(B^c_{R_0}(0))\bigcap\{|u|>\delta\}}2^t(|u_n-u|^{(q-1)t}|\varphi|^t+|u_n|^{(q-1)t}|\varphi|^t)dx\\
\leq&\xi C\|\varphi\|_{\varepsilon}+C_{\xi}2^t\int_{(B^c_{R_0}(0))\bigcap\{|u|>\delta\}}(|u_n-u|^t+|u_n|^t)|\varphi|^tdx
\end{split}
\end{equation*}
In view of $u\in H_{\varepsilon}$ we know that $\big|(\R^N\setminus B_{R}(0))\bigcap\{|u|>\delta\}\big|\rightarrow0$ in $R\rightarrow\infty,$ then there exists $R_1>0$ $s.t.$ $\big|(\R^N\setminus B_{R_1}(0))\bigcap\{|u|>\delta\}\big|<\xi.$ we define $R_2=\max\{R_0,R_1\},$ we deduce that
\begin{equation*}
\begin{split}
&\int_{(B_{R_2}^c(0))\bigcap\{|u|>\delta\}}|u_n-u|^t|\varphi|^tdx\\
\leq&\big(\int_{(B^c_{R_2}(0))\bigcap\{|u|>\delta\}}|u_n-u|^{t\frac{2^*_s}{t}}dx\big)^{\frac{t}{2^*_s}}\big(\int_{(B^c_{R_2}(0))\bigcap\{|u|>\delta\}}|\varphi|^{t\frac{2^*_s}{t}}dx\big)^{\frac{t}{2^*_s}}\big(\int_{(B^c_{R_2}(0))\bigcap\{|u|>\delta\}}1dx\big)^{\frac{4s-\mu}{2N-\mu}}\\
\leq&C\xi^{\frac{4s-\mu}{2N-\mu}}\|\varphi\|^t_{\varepsilon}
\end{split}
\end{equation*}
In similar way, we can prove that the follow inequality is true.
$$\int_{(B^c_{R_2}(0))\bigcap\{|u|>\delta\}}|u_n|^t|\varphi|^tdx\leq C\xi^{\frac{2s-\mu}{2N-\mu}}\|\varphi\|^t_{\varepsilon}$$
Hence,
$$\int_{B^c_{R_2}(0)}|f(w_n)+f(u)-f(u_n)|^t|\varphi|^tdx\leq\xi C\|\varphi\|_{\varepsilon}^t.$$
It is easy to verify that
\begin{equation}\label{eq3.17}
\int_{B_{R_2}(0)}|f(w_n)+f(u)-f(u_n)|^t|\varphi|^tdx\leq C\xi\|\varphi\|_{\varepsilon}.
\end{equation}
Since $u_n$ is bounded, we have $u_n\rightarrow u$ $a.e.$ in $\R^N,$ $u_n\rightarrow u$ in $L^p_{loc}(\R^N),$ $p\in[1,2^*_s).$ Let $l>q,$ such that $lt\in(2,2^*_s),$ $\frac{l}{l-1}t(q-1)\in(2,2^*_s).$ And
$$1<\frac{l}{l-1}t\leq\frac{q}{q-1}t=(1+\frac{1}{q-1})t<(1+\frac{N}{N-\mu})\cdot\frac{2N}{2N-\mu}<\frac{2N}{N-2s}=2^*_s.$$
Hence, we have
\begin{equation*}
\begin{split}
&|f(u_n-u)-f(u_n)+f(u)|^{t\frac{l}{l-1}}\\
\leq&C(|u_n-u|+|u_n-u|^{q-1}+|u_n|+|u_n|^{q-1}+|u|+|u|^{q-1})^{t\frac{l}{l-1}}\\
\leq&C(|u_n-u|^{t\frac{l}{l-1}}+|u_n-u|^{t\frac{l}{l-1}(q-1)}+|u_n|^{t\frac{l}{l-1}}+|u_n|^{t\frac{l}{l-1}(q-1)}+|u|^{t\frac{l}{l-1}}+|u|^{t\frac{l}{l-1}(q-1)})\\
=&:Ch_n.
\end{split}
\end{equation*}
Thus,
\begin{equation*}
\begin{split}
2\int_{B_{R_2}(0)}C(|u|^{\frac{l}{l-1}t}+|u|^{\frac{l}{l-1}t(q-1)})dx&=\int_{B_{R_2}(0)}\lim\limits_{n\rightarrow\infty}(Ch_n-|f(u_n-u)-f(u_n)+f(u)|^{\frac{l}{l-1}t})dx\\
&\leq\lim\limits_{n\rightarrow\infty}\big[C\int_{B_{R_2}(0)}h_ndx-\int_{B_{R_2}(0)}|f(u_n-u)-f(u_n)+f(u)|^{t\frac{l}{l-1}}dx\big]\\
&=\lim\limits_{n\rightarrow\infty}C\int_{B_{R_2}(0)}h_ndx-\limsup\limits_{n\rightarrow\infty}\int_{B_{R_2}(0)}|f(u_n-u)-f(u_n)+f(u)|^{t\frac{l}{l-1}}dx\\
&=2\int_{B_{R_2}(0)}C(|u|^{\frac{l}{l-1}t}+|u|^{\frac{l}{l-1}t(q-1)})dx\\
& \,\,\,\ -\limsup\limits_{n\rightarrow\infty}\int_{B_{R_2}(0)}|f(u_n-u)-f(u_n)+f(u)|^{t\frac{l}{l-1}}dx
\end{split}
\end{equation*}
Consequently,
$$0\leq\liminf\limits_{n\rightarrow\infty}\int_{B_{R_2}(0)}|f(u_n-u)-f(u_n)+f(u)|^{t\frac{l}{l-1}}dx\leq\limsup\limits_{n\rightarrow\infty}\int_{B_{R_2}(0)}|f(u_n-u)-f(u_n)+f(u)|^{t\frac{l}{l-1}}dx\leq0$$
So, we obtain \eqref{eq3.17} hold. By applying H\"{o}lder inequality, for any $\xi>0,$ $n$ large enough, we have
\begin{equation*}
\begin{split}
&\int_{B_{R_2}(0)}|f(u_n-u)-f(u_n)+f(u)|^t|\varphi|^tdx\\
\leq&\big(\int_{B_{R_2}(0)}|f(u_n-u)-f(u_n)+f(u)|^{t\frac{l}{l-1}}dx\big)^{\frac{l-1}{l}}\big(\int_{B_{R_2}(0)}|\varphi|^{tl}dx\big)^{\frac{1}{l}}\\
<&\xi\|\varphi\|_{\varepsilon}^t=\xi
\end{split}
\end{equation*}
As a consequence, $\int_{\R^N}|f(u_n-u)-f(u_n)+f(u)|^t|\varphi|^tdx\leq\xi\|\varphi\|^t_{\varepsilon}=\xi.$\\
$(iv)$
\begin{equation*}
\begin{split}
&\int_{\R^N}(\frac{1}{|x|^{\mu}}\ast F(u_n-u))f(u_n-u)\varphi dx-\int_{\R^N}(\frac{1}{|x|^{\mu}}\ast F(u_n))f(u_n)\varphi dx+(\frac{1}{|x|^{\mu}}\ast F(u))f(u)\varphi dx\\
=&\int_{\R^N}(\frac{1}{|x|^{\mu}}\ast F(u_n-u))(f(u_n-u)-f(u_n)+f(u))\varphi dx+\int_{\R^N}(\frac{1}{|x|^{\mu}}\ast F(u))(f(u)-f(u_n))\varphi dx\\
&+\int_{\R^N}(\frac{1}{|x|^{\mu}}\ast(F(u_n-u)-F(u_n)+F(u)))f(u)\varphi dx-\int_{\R^N}(\frac{1}{|x|^{\mu}}\ast F(u_n-u))f(u)\varphi dx\\
=&I_1+I_2+\int_{\R^N}(\frac{1}{|x|^{\mu}}\ast F(u))(f(u)-f(u_n)+f(u_n-u))\varphi dx\\
&-\int_{\R^N}(\frac{1}{|x|^{\mu}}\ast F(u))f(u_n-u)\varphi dx-\int_{\R^N}(\frac{1}{|x|^{\mu}}\ast F(u_n-u))f(u)\varphi dx
\end{split}
\end{equation*}
Clearly, we have
$$|I_1|\leq\big(\int_{\R^N}|F(u_n-u)|^tdx\big)^{\frac{1}{t}}\big(\int_{\R^N}|f(u_n-u)-f(u_n)+f(u)|^t|\varphi|^tdx\big)^{\frac{1}{t}}\leq C\xi\|\varphi\|_{\varepsilon}.$$
\begin{equation*}
\begin{split}
|I_2|&\leq(\int_{\R^N}|f(u_n)|^t|\varphi|^tdx)^{\frac{1}{t}}(\int_{\R^N}|F(u_n-u)-F(u_n)+F(u)|^t)^{\frac{1}{t}}\\
&\leq(\int_{\R^N}|f(u_n)|^{t\cdot\frac{l}{l-1}}dx)^{\frac{l-1}{lt}}(\int_{\R^N}|\varphi|^{lt}dx)^{\frac{1}{lt}}\xi\\
&\leq C\xi\|\varphi\|_{\varepsilon}.
\end{split}
\end{equation*}
In similar way, we get $|I_3|\leq C\xi\|\varphi\|_{\varepsilon}.$
Let us observe that,
\begin{equation*}
\begin{split}
&|\int_{\R^N}(\frac{1}{|x|^{\mu}}\ast F(u))f(u_n-u)\varphi dx|\\
\leq&C\int_{\R^N}(\frac{1}{|x|^{\mu}}\ast F(u))(|u_n-u|+|u_n-u|^{q-1})|\varphi|dx\\
\leq&C\int_{\R^N}(\frac{1}{|x|^{\mu}}\ast F(u))|u_n-u||\varphi|dx+C\int_{\R^N}(\frac{1}{|x|^{\mu}}\ast F(u))|u_n-u|^{q-1}|\varphi|dx\\
\leq&C(\int_{\R^N}|\varphi|^2dx)^{\frac{1}{2}}(\int_{\R^N}(\frac{1}{|x|^{\mu}}\ast F(u))^2|u_n-u|^2dx)^{\frac{1}{2}}+C(\int_{\R^N}|\varphi|^{qt}dx)^{\frac{1}{qt}}(\int_{\R^N}(\frac{1}{|x|^{\mu}}\ast F(u))^{\frac{qt}{qt-1}}|u_n-u|^{\frac{(q-1)qt}{qt-1}}dx)^{\frac{qt-1}{qt}}
\end{split}
\end{equation*}
Since $\frac{1}{|x|^{\mu}}\ast F(u)\in L^{\frac{2N}{\mu}}(\R^N),$ we have $(\frac{1}{|x|^{\mu}}\ast F(u))^2\in L^{\frac{N}{\mu}}(\R^N)$ and $(\frac{N}{\mu})'=\frac{N}{N-\mu},$ $\frac{2N}{N-\mu}\in(2,2^*_s).$ so, we have $|u_n-u|^2\in L^{\frac{N}{N-\mu}}(\R^N).$ Since $|u_n-u|^2\rightharpoonup0$ in $L^{\frac{N}{N-\mu}}(\R^N),$ we deduce that $$\int_{\R^N}(\frac{1}{|x|^{\mu}}\ast F(u))^2|u_n-u|^2\rightarrow0.$$
Moreover, we have $\frac{qt}{qt-1}\cdot\frac{qt-1}{qt}\cdot\frac{2N}{\mu}=\frac{2N}{\mu},$ $(\frac{2N(qt-1)}{qt\mu})'=\frac{2N(qt-1)}{(2N-\mu)qt-2N},$ $\frac{(q-1)qt}{qt-1}\cdot\frac{2N(qt-1)}{(2N-\mu)qt-2N}=qt\in(2,2^*_s)$ and $|u_n-u|^{\frac{(q-1)qt}{qt-1}}\rightharpoonup0$ in $L^{\frac{2N(qt-1)}{(2N-\mu)qt-2N}}(\R^N),$ we get $(\frac{1}{|x|^{\mu}}\ast F(u))^{\frac{qt}{qt-1}}\in L^{\frac{2N(qt-1)}{qt\mu}}(\R^N),$  $|u_n-u|^{\frac{(q-1)qt}{qt-1}}\in L^{\frac{2N(qt-1)}{(2N-\mu)qt-2N}}(\R^N)$ and $\int_{\R^N}(\frac{1}{|x|^{\mu}}\ast F(u))^{\frac{qt}{qt-1}}|u_n-u|^{\frac{(q-1)qt}{qt-1}}dx\rightarrow0.$
Hence, we have
$$|\int_{\R^N}(\frac{1}{|x|^{\mu}}\ast F(u))f(u_n-u)\varphi|\leq C\xi\|\varphi\|_{\varepsilon}$$
In similar way, $|\int_{\R^N}(\frac{1}{|x|^{\mu}}\ast F(u_n-u))f(u)\varphi|\leq C\xi\|\varphi\|_{\varepsilon}.$ Therefore $(iv)$ holds.
\end{proof}
\par
By using Brezis-Lieb Lemma \cite{Brezis-Lieb1983PAMS,Gao-Yang18SCI} and Lemma \ref{lem3.8}, we have the following lemma.
\begin{lemma}\label{lem3.9}
Let $\{u_n\}\subset H_{\varepsilon}$ be a $(PS)_d$ sequence of $J_{\varepsilon}$ with $u_n\rightharpoonup u$ in $H_{\varepsilon},$ then
\begin{itemize}
\item [$(i)$] $J_{\varepsilon}(w_n)=J_{\varepsilon}(u_n)-J_{\varepsilon}(u)+o_n(1),$
\item [$(ii)$] $\|J'_{\varepsilon}(w_n)\|=o_n(1).$
\end{itemize}
\end{lemma}
\begin{proof}
$(i)$ We note that
\begin{equation*}
\begin{split}
&J_{\varepsilon}(u_n-u)-J_{\varepsilon}(u_n)+J_{\varepsilon}(u)\\
=&\frac{1}{2}\big(\|u_n-u\|^2_{\varepsilon}-\|u_n\|_{\varepsilon}^2+\|u\|^2_{\varepsilon})-\frac{1}{2}(\int_{\R^N}(\frac{1}{|x|^{\mu}}\ast F((u_n-u)^+))F((u_n-u)^+)dx\\
&-\int_{\R^N}(\frac{1}{|x|^{\mu}}\ast F(u^+_n))F(u_n^+)dx+\int_{\R^N}(\frac{1}{|x|^{\mu}}\ast F(u))F(u)dx\big)-\frac{1}{2^*_s}\int_{\R^N}(|(u_n-u)^+|^{2^*_s}-|u_n^+|^{2^*_s}+|u^+|^{2^*_s})dx
\end{split}
\end{equation*}
By the Lemma\ref{lem3.8} $(ii)$, $u_n\rightharpoonup u$ in $H_{\varepsilon}$ and Brezis-Lieb Lemma. we have $(i)$ holds.\\
$(ii)$ Recall that $\{u_n\}$ is a $(PS)_d$ sequence of $J_{\varepsilon},$ we have $\|J'_{\varepsilon}(u_n)\|=o_n(1),$ $J'_{\varepsilon}(u)=0.$ For any $\xi>0,$ $n$ large enough, $\forall$ $\varphi\in H_{\varepsilon}$ and $\|\varphi\|_{\varepsilon}=1,$ by the lemma\ref{lem3.8} $(iv)$ we get
\begin{equation*}
\begin{split}
&|\langle J'_{\varepsilon}(u_n-u),\varphi\rangle|\\
=&\big|\langle J'_{\varepsilon}(u_n),\varphi\rangle-\langle J'_{\varepsilon}(u),\varphi\rangle-\int_{\R^N}(\frac{1}{|x|^{\mu}}\ast F((u_n-u)^+))F((u_n-u)^+)\varphi dx-\int_{\R^N}(\frac{1}{|x|^{\mu}}\ast F(u^+_n))f(u_n^+)\varphi dx\\
&-\int_{\R^N}(\frac{1}{|x|^{\mu}}\ast F(u^+))f(u^+)\varphi dx-\int_{\R^N}(|(u_n-u)^+|^{2^*_s-1}-|u_n^+|^{2^*_s-1}+|u^+|^{2^*_s-1})\varphi dx\big|\\
\leq&\|J'_{\varepsilon}(u_n)\|\|\varphi\|_{\varepsilon}+\big|\int_{\R^N}(\frac{1}{|x|^{\mu}}\ast F((u_n-u)^+))f((u_n-u)^+)\varphi dx-\int_{\R^N}(\frac{1}{|x|^{\mu}}\ast F(u_n^+))f(u_n^+)\varphi dx\\
&+\int_{\R^N}(\frac{1}{|x|^{\mu}}\ast F(u^+))f(u^+)\varphi dx\big|+\int_{\R^N}\big||(u_n-u)^+|^{2^*_s-1}-|u_n^+|^{2^*_s-1}+|u^+|^{2^*_s-1}\big||\varphi|dx\\
\leq&\xi\|\varphi\|_{\varepsilon}+C\xi\|\varphi\|_{\varepsilon}+\big(\int_{\R^N}||u_n-u|^{2^*_s-1}-|u_n^+|^{2^*_s-1}+|u^+|^{2^*_s-1}|^{\frac{2^*_s}{2^*_s-1}}dx\big)^{\frac{2^*_s-1}{2^*_s}}\|\varphi\|_{\varepsilon}\\
\leq&\xi\|\varphi\|_{\varepsilon}+C\xi\|\varphi\|_{\varepsilon}+\xi\|\varphi\|_{\varepsilon}
\end{split}
\end{equation*}
This completes the proof of $(ii)$
\end{proof}
\begin{lemma}\label{lem3.10}
$J_{\varepsilon}$ satisfies the $(PS)_d$ condition at any level $d\leq m_{_{V_{\infty}}}.$
\end{lemma}
\begin{proof}
Let $u_n\subset H_{\varepsilon}$ be a $(PS)_d$ sequence of $J_{\varepsilon}.$ Then, by Lemma \ref{Lem3.3} we know that $\{u_n\}$ is bounded in $H_{\varepsilon}$ and we can assume $u_n\geq0.$ Hence, up to a subsequence, there is $u\in H_{\varepsilon}$ such that $u_n\rightharpoonup u\geq0$ in $H_{\varepsilon},$ $u_n\rightarrow u$ in $L^r_{loc}(\R^N)$ for each $r\in[2,2^*_s),$ $u_n(x)\rightarrow u(x)$ $a.e.$ in $\R^N$ and $J'_{\varepsilon}(u)=0.$ Set $w_n=u_n-u,$ by Lemma \ref{lem3.8} we have $$J_{\varepsilon}(w_n)=J_{\varepsilon}(u_n)-J_{\varepsilon}(u)+o_n(1)=d-J_{\varepsilon}(u)+o_n(1) \,\,\ \text{ and } \,\,\ J'_{\varepsilon}(w_n)=o_n(1).$$ Moreover, for any $\alpha\in(2,2^*_s)$ and $\alpha\leq4,$ we have
\begin{equation*}
\begin{split}
J_{\varepsilon}(u)&=J_{\varepsilon}(u)-\frac{1}{\alpha}\langle J'_{\varepsilon(u),u}\rangle\\
&=(\frac{1}{2}-\frac{1}{\alpha})\|u_n\|^2_{\varepsilon}+\frac{1}{\alpha}\int_{\R^N}(\frac{1}{|x|^{\mu}}\ast F(u_n))f(u_n)u_ndx-\frac{1}{2}\int_{\R^N}(\frac{1}{|x|^{\mu}}\ast F(u_n))F(u_n)dx-\frac{1}{2^*_s}\int_{\R^N}|u_n|^{2^*_s}dx\\
& \,\,\, \,\,\ +\frac{1}{\alpha}\int_{\R^N}|u_n|^{2^*_s-2}u_n^2dx\\
&\geq0.
\end{split}
\end{equation*}
By Lemma \ref{Lem3.5} we have $d-J_{\varepsilon}(u)\leq d\leq m_{V_{\infty}}<\frac{s}{N}S^{\frac{N}{2s}}$ and by Lemma \ref{lem3.7} we know $u_n\rightarrow u$ in $H_{\varepsilon}.$ Hence, the Lemma is proved.
\end{proof}
\\
By Lemma \ref{lem3.3.1} and Lemma \ref{lem3.10} we have the following lemma.
\begin{lemma}\label{lem3.11}
$J_{\varepsilon}|_{\mathcal{N}_{\varepsilon}}$ satisfies the $(PS)_d$ condition at any level $d<m_{V_\infty}.$
\end{lemma}

$\mathbf{Proof\ of\ Theorem~\ref{Thm1.1}}.$
By Lemma \ref{Lem3.1} we know that functional $J_{\varepsilon}$ satisfies the mountain pass geometry, then using a version of the mountain pass theorem, there exists a sequence $\{u_n\}\subset H_{\varepsilon}$ such that
$$\lim\limits_{n\rightarrow\infty}J_{\varepsilon}=c_{\varepsilon} \,\,\ \text{ and } \,\,\ (1+\|u_n\|_{\varepsilon})\|J'_{\varepsilon}\|=o_n(1).$$
For any $\tau\in\R$ with $V_0<\tau<V_{\infty},$ we have $m_{V_0}<m_{\tau}<m_{V_{\infty}}.$ By Lemma \ref{Lem3.5}, $m_{\tau}<\frac{s}{N}S^{\frac{N}{2s}}.$ Apply Lemma \ref{Lem3.3}, Lemma \ref{lem3.10} and Theorem 6.3.4 in \cite{Zhongbook}, we obtain that $m_{\tau}$ is a critical value of $I_{\tau}$ with corresponding nontrivial nonnegative critical point $u\in H_{\varepsilon}(\R^N).$ For any $r>0,$ take $\eta_{r}\in C^{\infty}_0(\R^N,[0,1])$ be such that $$\eta_r=1 \text{ if } |x|<r \,\,\ \text{ and } \,\,\ \eta_r=0 \text{ if } |x|>2r.$$ Set $u_r:=\eta_ru,$ it is easy to verify that $u_r\in H_{\varepsilon}(\R^N)$ for each $r>0.$ By Lemma \ref{Lem3.2} there exists $t_r>0$ such that $\tilde{u}_r:=t_ru_r\in\mathcal{M}_{\tau}.$ Consequently, there is $r_0>0$ such that $\tilde{u}:=\tilde{u}_{r_0}$ satisfies $I_{\tau}(\tilde{u})<m_{V_{\infty}}.$ In fact, if this is false, then $I_{\tau}(\tilde{u}_r)=I_{\tau}(t_ru_r)\geq m_{V_{\infty}}$ for all $r>0.$ Notice that $u_r\rightarrow u$ in $H_{\varepsilon}(\R^N)$ as $r\rightarrow+\infty$ and $u\in\mathcal{M}_{\tau}.$ we can deduce that $t_r\rightarrow1$ as $r\rightarrow+\infty.$ Hence,$$m_{_{V_{\infty}}}\leq\liminf\limits_{r\rightarrow+\infty}I_{\tau}(t_ru_r)=I_{\tau}(u)=m_{\tau}<m_{_{V_{\infty}}},$$
which gives a contradiction, then $I_{\tau}(\tilde{u})<m_{_{V_{\infty}}}.$ The invariance by translation, we may assume $V_0=V(0)<\tau$ and $supp(\tilde{u})$ is compact. We use the continuity of $V,$ there is an $\varepsilon^*>0$ such that $$V(\varepsilon x)<\tau, \,\,\ \forall \varepsilon\in(0,\varepsilon^*) \,\,\ \text{and} \,\,\ x\in supp(\tilde{u}).$$
Hence, $$J_{\varepsilon}(t\tilde{u})\leq I_{\tau}(t\tilde{u}), \,\,\ \forall\varepsilon\in(0,\varepsilon^*) \,\,\ \text{and} \,\,\ t\geq0,$$
and $$\max\limits_{t\geq0}J_{\varepsilon}(t\tilde{u})\leq\max\limits_{t\geq0}I_{\tau}(t\tilde{u})=I_{\tau}(\tilde{u})<m_{_{V_{\infty}}}, \,\,\ \forall \,\,\ \varepsilon\in(0,\varepsilon*).$$
Consequently, $$c_{\varepsilon}<m_{_{V_{\infty}}}, \,\,\ \forall \,\,\ \varepsilon\in(0,\varepsilon^*).$$
Lemma \ref{lem3.10} guarantees up to a subsequence such that $u_n\rightarrow u$ in $H_{\varepsilon},$ then $J'_{\varepsilon}(u)=0$ and $J_{\varepsilon}(u)=c_{\varepsilon}.$ Hence $u$ is a ground nontrivial nonnegative solution of \eqref{eq2.3}. This completes the proof of Theorem \ref{Thm1.1}.
\raisebox{-0.5mm}{\rule{1.5mm}{1mm}}\vspace{6pt}
\section{Multiplicity Results}
\subsection{Technical results}
In this section we focus our attention on the study of the multiplicity of solutions to \eqref{eq1.1}. Since $V_0>0$, by Lemma \ref{Lem3.5}, $m_{_{V_0}}<\frac{s}{N}S^{\frac{N}{2s}}.$ From the proof of Theorem \ref{Thm1.1} we know that $m_{_{V_0}}$ is a critical value of $I_{V_0}$ with corresponding nontrivial nonnegative critical point $w\in H^s(\R^N).$ Fix $\delta>0$ and let $\eta\in C^{\infty}(\R^+,[0,1])$ be a function such that $\eta(t)=1$ if $0\leq t\leq\frac{\delta}{2}$ and $\eta(t)=0$ if $t\geq\delta.$ For any $y\in\Lambda,$ we define
$$\Psi_{\varepsilon,y}(x)=\eta(|\varepsilon x-y|)w(\frac{\varepsilon x-y}{\varepsilon}),\forall x\in\R^N.$$
Then for small $\varepsilon>0$, one has $\Psi_{\varepsilon,y}\in H_\varepsilon\backslash\{0\}$ for all $y\in \Lambda$. In fact, using the change of variable $z=x-\frac{y}{\varepsilon}$, one has
\begin{equation*}
\begin{split}
\int_{\mathbb{R}^N}V(\varepsilon x)\Psi_{\varepsilon,y}^2(x)dx&=\int_{\mathbb{R}^N}V(\varepsilon x)\eta^2(|\varepsilon x-y|)w^2(\frac{\varepsilon x-y}{\varepsilon})dx=\int_{\mathbb{R}^N}V(\varepsilon z+y)\eta^2(|\varepsilon z|)w^2(z)dz\\
&\leq C\int_{\mathbb{R}^N}w^2(z)dz<+\infty.
\end{split}
\end{equation*}
Moreover, using the change of variable $x^\prime=x-\frac{y}{\varepsilon},z^\prime=z-\frac{y}{\varepsilon}$, we have
\begin{equation*}
\begin{split}
\|(-\Delta)^{\frac{s}{2}}\Psi_{\varepsilon,y}\|_{L^2(\R^N)}^2&=\frac{1}{2}C(s)\iint_{\mathbb{R}^N\times \mathbb{R}^N}\frac{\big|\eta(|\varepsilon x-y|)w(\frac{\varepsilon x-y}{\varepsilon})-\eta(|\varepsilon z-y|)w(\frac{\varepsilon z-y}{\varepsilon})\big|^2}{|x-z|^{N+2s}}dxdz\\
&=\frac{1}{2}C(s)\iint_{\mathbb{R}^N\times \mathbb{R}^N}\frac{\big|\eta(|\varepsilon x^\prime|)w(x^\prime)-\eta(|\varepsilon z^\prime|)w(z^\prime)\big|^2}{|x^\prime-z^\prime|^{N+2s}}dx^\prime dz^\prime\\
&=\|(-\Delta)^{\frac{s}{2}}\eta(|\varepsilon x|)w(x)\|_{L^2(\R^N)}^2=\|(-\Delta)^{\frac{s}{2}}\eta_\varepsilon w\|_{L^2(\R^N)}^2,
\end{split}
\end{equation*}
where $\eta_\varepsilon(x)=\eta(|\varepsilon x|)$. By Lemma \ref{Lem2.4}, we see that $\eta_\varepsilon w\in \mathcal{D}^{s,2}(\mathbb{R}^N)$ as $\varepsilon\rightarrow 0$, and hence $\Psi_{\varepsilon,y}\in \mathcal{D}^{s,2}(\mathbb{R}^N)$ for $\varepsilon>0$ small. Hence $\Psi_{\varepsilon,y}\in H_\varepsilon$. Now we proof $\Psi_{\varepsilon,y}\neq 0$. In fact,
\begin{equation*}
\begin{split}
\int_{\mathbb{R}^N}\Psi_{\varepsilon,y}^2(x)dx&=\int_{\mathbb{R}^N}\eta^2(|\varepsilon x-y|)w^2(\frac{\varepsilon x-y}{\varepsilon})dx=\int_{|\varepsilon x-y|<\delta}\eta^2(|\varepsilon x-y|)w^2(\frac{\varepsilon x-y}{\varepsilon})dx\\
&\geq\int_{|z|\leq\frac{\delta}{2\varepsilon}}\eta^2(|\varepsilon z|)w^2(z)dz\geq\int_{B_0(\frac{\delta}{2\varepsilon})}w^2(z)dz\rightarrow \int_{\mathbb{R}^N}w^2(z)dz>0
\end{split}
\end{equation*}
as $\varepsilon\rightarrow 0$. Then $\Psi_{\varepsilon,y}\neq 0$ for small $\varepsilon>0$. Therefore, there exists unique $t_\varepsilon>0$ such that
\begin{equation*}
\max_{t\geq 0}I_\varepsilon(t \Psi_{\varepsilon,y})=I_\varepsilon(t_\varepsilon \Psi_{\varepsilon,y})\ \text{and}\ t_\varepsilon\Psi_{\varepsilon,y}\in\mathcal{N}_\varepsilon.
\end{equation*}
We introduce the map $\Phi_\varepsilon:\Lambda\rightarrow \mathcal{N}_\varepsilon$ by setting
\begin{equation*}
\Phi_\varepsilon(y)=t_\varepsilon \Psi_{\varepsilon,y}.
\end{equation*}
By construction, $\Phi_\varepsilon(y)$ has a compact support for any $y\in \Lambda$ and $\Phi_\varepsilon$ is a continuous map.

\begin{lemma}\label{lem4.1}
$$\lim\limits_{\varepsilon\rightarrow0}J_{\varepsilon}(\Phi_{\varepsilon}(y))=m_{_{V_0}} \,\,\ \text{ uniformly in } \,\,\ y\in\Lambda.$$
\end{lemma}
\begin{proof}
Assume by contradiction, then there exists $\delta_0>0,$ $\{y_n\}\subset\Lambda$ and $\varepsilon_n>0$ with $\varepsilon_n\rightarrow0$ such that
\begin{equation}\label{eq4.1}
|J_{\varepsilon_n}(\Phi_{\varepsilon_n}(y_n))-m_{_{V_0}}|\geq\delta_0.
\end{equation}
By using $\Phi_{\varepsilon_n}\in\mathcal{N}_{\varepsilon_n}$ and Lemma \ref{Lem3.4} we know that there is a $r_0>0$ such that
\begin{equation}\label{eq4.2}
\begin{split}
&\int_{\R^N}(\frac{1}{|x|^{\mu}}\ast F(\Phi_{\varepsilon_n}(y_n)))f(\Phi_{\varepsilon_n}(y_n))\Phi_{\varepsilon_n}(y_n)dx+\int_{\R^N}|\Phi_{\varepsilon_n}(y_n)|^{2^*_s}dx\\
=&\|\Phi_{\varepsilon_n}(y_n)\|^2_{\varepsilon_n}\\
\geq& r_0
\end{split}
\end{equation}
which implies that $t_{\varepsilon}\nrightarrow0.$ Hence there exists a $T>0$ such that $t_{\varepsilon_n}\geq T.$ If $t_{\varepsilon_n}\rightarrow\infty,$ we have
\begin{equation*}
\begin{split}
C\|w\|^2_{\varepsilon}&\geq\int_{\R^N}|(-\Delta)^{\frac{1}{2}}\Psi_{\varepsilon_n,y_n}|^2dx+\int_{\R^N}V(\varepsilon_nx)\Psi_{\varepsilon_n,y_n}^2dx\\
&=t^{-2}_{\varepsilon_n}\int_{\R^N}(\frac{1}{|x|^{\mu}}\ast F(\Phi_{\varepsilon_n}(y_n)))f(\Phi_{\varepsilon_n}(y_n))\Phi_{\varepsilon_n}(y_n)dx+t^{-2}_{\varepsilon_n}\int_{\R^N}|\Phi_{\varepsilon_n}(y_n)|^{2^*_s}dx\\
&\geq t^{-2}_{\varepsilon_n}\int_{\R^N}|t_{\varepsilon_n}\Psi_{\varepsilon_n,y_n}|^{2^*_s}dx\\
&\geq t^{-2}_{\varepsilon_n}\int_{\R^N}|t_{\varepsilon_n}\eta(|\varepsilon_nx|)w(x)|^{2^*_s}dx\\
&\geq t_{\varepsilon_n}^{-2}\int_{|x|<\frac{\delta}{2\varepsilon_n}}|t_{\varepsilon_n}w(x)|^{2^*_s}dx\\
&\geq t_{\varepsilon_n}^{2^*_s-2}\int_{\frac{\delta}{2}<|x|<\delta}|w(x)|^{2^*_s}dx\\
&\rightarrow+\infty
\end{split}
\end{equation*}
for large $n.$ This yield a contradiction, then $t_{\varepsilon}\rightarrow t_0>0.$ Now we claim that $t_0\rightarrow1.$ By using Lebesgue's theorem, we can verify that $$\lim\limits_{n\rightarrow\infty}\|\Phi_{\varepsilon_n}(y_n)\|^2_{\varepsilon}=t_0^2\|w\|^2_{V_0},$$
$$\lim\limits_{n\rightarrow\infty}\int_{\R^N}(\frac{1}{|x|^{\mu}}\ast F(\Phi_{\varepsilon_n}(y_n)))f(\Phi_{\varepsilon_n}(y_n))\Phi_{\varepsilon_n}(y_n)dx=\int_{\R^N}(\frac{1}{|x|^{\mu}}\ast F(t_0w))f(t_0w)t_0wdx,$$
and $$\lim\limits_{n\rightarrow\infty}\int_{\R^N}|\Phi_{\varepsilon_n}(y_n)|^{2^*_s}dx=\int_ {\R^N}|t_0w|^{2^*_s}dx.$$
Therefore, from \eqref{eq4.2}, we get
$$t_0^2\|w\|^2_{V_0}=\int_{\R^N}(\frac{1}{|x|^{\mu}}\ast F(t_0w))f(t_0w)t_0wdx+\int_{\R^N}|t_0w|^{2^*_s}dx.$$
This show $t_0w\in\mathcal{M}_{V_0}.$ Noting that $w\in\mathcal{M}_{V_0},$ we see $t_0=1,$ so claim is proved. Moreover, similar to the above arguments, we can get
$$\lim\limits_{n\rightarrow\infty}J_{\varepsilon_n}(\Phi_{\varepsilon_n}(y_n))=I_{V_0}(w)=m_{_{V_0}}$$
which contradicts to \eqref{eq4.1}. This completes the proof.
\end{proof}
\par
Now, we are ready to introduce the barycenter map. For any $\delta>0,$ let $\rho=\rho(\delta)>0$ such that $\Lambda_{\delta}\subset B_{\rho}(0).$ Define $\Upsilon:\R^N\rightarrow\R^N$ as follow:
\begin{align*}
\Upsilon(x)=\left\{ \begin{array}{ll}
x \,\,\ \,\,\ \,\,\ \text{if} \,\,\ |x|<\rho\\
\frac{\rho x}{|x|} \,\,\ \text{if} \,\,\ |x|\geq\rho\\
\end{array} \right.
\end{align*}
We define the barycenter map $\beta_{\varepsilon}:\mathcal{N}_{\varepsilon}\rightarrow\R^N$ as follows
$$\beta_{\varepsilon}=\frac{\int_{\R^N}\Upsilon(\varepsilon x)|w(x)|^2dx}{\int_{\R^N}|w(x)|^2dx}.$$
\begin{lemma}\label{lem4.2}
$$\lim\limits_{\varepsilon\rightarrow0}\beta_{\varepsilon}(\Phi_{\varepsilon}(y))=y \,\,\ \text{ uniformly in } \,\,\ y\in\Lambda.$$
\end{lemma}
\begin{proof}
Assume by contradiction, then there exists $\delta_0>0,$ $\{y_n\}\subset\Lambda$ and $\varepsilon_n\rightarrow0^+$ such that
\begin{equation}\label{eq4.3}
|\beta_{\varepsilon_n}(\Phi_{\varepsilon_n}(y_n))-y_n|\geq\delta_0>0, \,\,\ \forall n\in \N.
\end{equation}
By using the definitions of $\beta_{\varepsilon_n}$ and $\Phi_{\varepsilon_n},$ we can see that
$$\beta_{\varepsilon_n}(\Phi_{\varepsilon_n}(y_n))=y_n+\frac{\int_{\R^N}[\Upsilon(\varepsilon_nx+y_n)-y_n]|\eta(|\varepsilon_nx|)w(x)|^2dx}{\int_{\R^N}|\eta(|\varepsilon_nx|)w(x)|^2dx}.$$
Taking into account the Lebesgue dominant convergence theorem, we can infer  that
$$|\beta_{\varepsilon_n}(\Phi_{\varepsilon_n}(y_n))-y_n|\rightarrow0$$
which contradicts \eqref{eq4.3}.
\end{proof}
\begin{lemma}\label{lem4.3}
For any $\tau>0,$ let $\{u_n\}\subset\mathcal{M}_{\tau}$ with $I_{\tau}(u_n)\rightarrow m_{\tau}.$ Then $\{u_n\}$ has a subsequence strongly convergent in $H^s(\R^N).$ Particulary, there exists a minimizer for $m_{\tau}.$
\end{lemma}
\begin{proof}
From the proof of Lemma \ref{Lem3.3} and Lemma \ref{Lem3.5}, we know that $\{u_n\}$ is bounded in $H^s(\R^N)$ and $m_{\tau}<\frac{s}{N}S^{\frac{N}{2s}}.$ By the Ekeland Variational principle, we may assume that $\{u_n\}$ is a $(PS)_{m_{\tau}}$ sequence of $I_{\tau}.$ Then, by Lemma \ref{lem3.7}, there exists $u\in H^s(\R^N)$ such that, up to a subsequence, $u_n\rightarrow u$ in $H^s(\R^N).$ Moreover, $u$ is a minimizer of $m_{\tau}.$
\end{proof}
\begin{lemma}\label{lem4.4}
Let $\varepsilon_n\rightarrow0$ and $u_n\in\mathcal{N}_{\varepsilon_n}$ be such that $J_{\varepsilon_n}(u_n)\rightarrow m_{_{V_0}}.$ Then there exists a sequence $\{y_n\}\subset\R^N$ such that $u_n(\cdot+y_n)$ has a convergent subsequence in $H^s(\R^N).$ Moreover, up to a subsequence, $\tilde{y}_n=\varepsilon_ny_n\rightarrow y\in\Lambda.$
\end{lemma}
\begin{proof}
Since $u_n\in\mathcal{N}_{\varepsilon_n}$ and $\lim\limits_{n\rightarrow\infty}J_{\varepsilon_n}(u_n)=m_{_{V_0}},$ by Lemma \ref{Lem3.3} we can see that $\{u_n\}$ is bounded in $H^s(\R^N).$ By Lemma \ref{Lem3.4}, we have $\|u_n\|_{\varepsilon_n}\nrightarrow0.$ we can argue as in Lemma \ref{Lem3.6} to obtain a sequence $\{y_n\}$ and constant $r>0$ such that
\begin{equation}\label{eq4.4}
\liminf\limits_{n\rightarrow\infty}\int_{B_r(y_n)}|u_n(x)|^2dx=\beta>0.
\end{equation}
Note, if this is false, then  for any $r>0,$ we have
$$\lim\limits_{n\rightarrow\infty}\sup\limits_{y\in\R^N}\int_{B_r(y)}|u_n|^2dx=0.$$
By Lemma \ref{Lem2.2}, we know that $u_n\rightarrow0$ in $L^t(\R^N)$ for $t\in[2,2^*_s),$ we can argue as  the proof of \eqref{eq3.2} and we deduce that $$\int_{\R^N}(\frac{1}{|x|^{\mu}}\ast F(u))f(u)udx=o_n(1).$$
As the proof of Lemma \ref{Lem3.6}, we can prove $\int_{\R^N}|u|^{2^*_s}dx=o_n(1).$ Since $u_n\in\mathcal{N}_{\varepsilon_n},$ we get $\|u_n\|_{\varepsilon_n}=o_n(1),$ which gives a contradiction. Hence, \eqref{eq4.4} holds. Now, we set $\tilde{u}_n=u_n(\cdot+y_n).$ Since, $\{u_n\}$ is bounded in $H^s(\R^N)$ and \eqref{eq4.4}, up to a subsequence, we have $\tilde{u}_n\rightharpoonup\tilde{u}\neq0$ in $H^s(\R^N)$ and $\tilde{u}_n(x)\rightarrow\tilde{u}(x)$ $a.e.$ in $\R^N.$ Fix $t_n>0$ such that $t_n\tilde{u}_n\in \mathcal{M}_{V_0}$ and set $\tilde{y}_n=\varepsilon_ny_n.$
Since $u_n\in\mathcal{N}_{\varepsilon_n},$ we can see that
\begin{equation*}
\begin{split}
m_{_{V_0}}&\leq I_{V_0}(t_n\tilde{u}_n)\\
&=\frac{1}{2}t^2_n[\tilde{u}_n]^2+\frac{t^2_n}{2}\int_{\R^N}V_0\tilde{u}_n^2dx-\frac{1}{2^*_s}\int_{\R^N}|t_n\tilde{u}_n^+|^{2^*_s}dx-\frac{1}{2}(\frac{1}{|x|^{\mu}}\ast F(t_n\tilde{u}_n^+))F(t_n\tilde{u}_n^+)dx\\
&\leq J_{\varepsilon_n}(t_nu_n)\\
&\leq J_{\varepsilon_n}(u_n)\\
&=m_{_{V_0}}+o_n(1).
\end{split}
\end{equation*}
which gives $$\lim\limits_{n\rightarrow\infty}I_{V_0}(t_n\tilde{u}_n)=m_{_{V_0}}>0.$$
By Lemma \ref{lem4.3}, up to subsequence, we get $t_n\tilde{u}_n:=v_n\rightarrow v_0$ in $H^s(\R^N).$ Note,
\begin{equation*}
\beta=\liminf\limits_{n\rightarrow\infty}\int_{B_r(y_n)}|u_n(x)|^2dx=\liminf\limits_{n\rightarrow\infty}\int_{B_r(0)}|\tilde{u}_n(x)|^2dx\leq\liminf_{n\rightarrow\infty}\|\tilde{u}_n\|^2_{H^s(\R^N)}.
\end{equation*}
For large $n,$ we have $0<\frac{\beta}{2}<\|\tilde{u}_n\|^2_{H^s(\R^N)},$ then $$0<\frac{\beta}{2}t_n^2<\|t_n\tilde{u}_n\|^2_{H^s(\R^N)}=\|v_n\|^2_{H^s(\R^N)}\leq C.$$
Hence $\{t_n\}$ is bounded, and we may assume that $t_n\rightarrow t^*>0.$ So, up to a subsequence, we have $$v_n\rightarrow v_0=t^*\tilde{u}\neq0 \text{ in } H^s(\R^N), \,\,\  \,\,\ \tilde{u}_n\rightarrow\frac{1}{t^*}v_0=\tilde{u} \text{ in } H^s(\R^N).$$
In order to complete the proof of the lemma, we show that $\{\tilde{y}_n\}$ is bounded in $\R^N.$ We argue by contradiction, up to a subsequence, we assume that $|\tilde{y}_n|\rightarrow\infty.$ Notice that, up to subsequence, we have $v_n\rightarrow v_0\neq0$ in $H^s(\R^N).$
By Fatou's lemma we get
\begin{equation*}
\begin{split}
m_{_{V_0}}&=I_{V_0}(v_0)\\
&<I_{V_{\infty}}(v_0)-\frac{1}{2}\langle I'_{V_0}(v_0),v_0\rangle\\
&=\frac{1}{2}\int_{\R^N}(V _{\infty}v_0^2-V_0v_0^2)dx-\frac{1}{2}\int_{\R^N}(\frac{1}{|x|^{\mu}}\ast F(v_0^+))F(v_0^+)dx+\frac{1}{2}\int_{\R^N}(\frac{1}{|x|^{\mu}}\ast F(v_0^+))f(v_0^+)v^+_0dx\\
& \,\,\ \,\,\ -\frac{1}{2^*_s}\int_{\R^N}|v_0^+|^{2^*_S}dx+\frac{1}{2}\int_{\R^N}|v_0^+|^{2^*_s}dx\\
&\leq\liminf\limits_{n\rightarrow\infty}(J_{\varepsilon_n}(v_n)-\frac{1}{2}\langle I'_{V_0}(v_n),v_n\rangle)\\
&=\liminf\limits_{n\rightarrow\infty}J_{\varepsilon_n}(v_n)\\
&\leq\lim\limits_{n\rightarrow\infty}J_{\varepsilon_n}(u_n)=m_{_{V_0}}
\end{split}
\end{equation*}
which is a contradiction, so we get $\{\tilde{y}_n\}$ is bounded in $\R^N.$ Therefore, up to subsequence, $\tilde{y}_n\rightarrow y\in\R^N.$ If $y\in\R^N\setminus\Lambda$ then $V_0<V(y).$ This is a contradiction. Hence, we can conclude that $y\in\Lambda.$
\end{proof}
\par
Now, we introduce a subset $\tilde{\mathcal{N}}_{\varepsilon}$ of $\mathcal{N}_{\varepsilon}$ by setting
$$\tilde{\mathcal{N}}_{\varepsilon}=\{u\in\mathcal{N}_{\varepsilon}:J_{\varepsilon}(u)\leq m_{_{V_0}}+h(\varepsilon)\},$$
where $h(\varepsilon):=\max\limits_{y\in\Lambda}|J_{\varepsilon}(\Phi_{\varepsilon}(y))-m_{_{V_0}}|.$ Then, we can use Lemma \ref{lem4.1} to conclude that $$\lim\limits_{\varepsilon\rightarrow0^+}h(\varepsilon)=0.$$
\par
Hence, for each $y\in\Lambda$ and $\varepsilon>0,$ we have $\Phi_{\varepsilon}(y)\in\tilde{\mathcal{N}}_{\varepsilon}.$
By Lemma \ref{lem4.4}, we can prove the following Lemma.
\begin{lemma}\label{lem4.5}
For any $\delta>0,$ we have
$$\lim\limits_{\varepsilon\rightarrow0}\sup\limits_{u\in\tilde{\mathcal{N}}_{\varepsilon}}dist(\beta_{\varepsilon}(u),\Lambda_{\delta})=0.$$
\end{lemma}
\begin{proof}
Let $\varepsilon_n\rightarrow0.$ For any $n\in\N,$ there exists $\{u_n\}\subset\mathcal{\tilde{N}}_{\varepsilon_n}$ such that
$$\inf\limits_{y\in\Lambda_{\delta}}|\beta_{\varepsilon_n}(u_n)-y|=\sup\limits_{u\in\tilde{\mathcal{N}}}\inf\limits_{y\in\Lambda_{\delta}}|\beta_{\varepsilon_n}(u)-y|+o_n(1).$$
Since $\{u_n\}\in\mathcal{N}_{\varepsilon_n},$ it follow that $$m_{_{V_0}}\leq c_{\varepsilon_n}\leq J_{\varepsilon_n(u_n)}\leq m_{_{V_0}}+h(\varepsilon_n).$$
Then, $J_{\varepsilon_n}(u_n)\rightarrow m_{_{V_0}}.$ By Lemma \ref{lem4.4}, there exists $\{y_n\}\in\R^N$ such that $\{\tilde{u}_n(\cdot):=u_n(\cdot+y_n)\}$ has a convergent subsequence in $H^s(\R^N)$ and $\tilde{y}_n:=\varepsilon_ny_n\rightarrow y\in\Lambda.$ Then, $$\beta_{\varepsilon_n}(u_n)=\tilde{y}_n+\frac{\int_{\R^N}[\chi(\varepsilon_nx+\tilde{y}_n)-\tilde{y}_n]|\tilde{u}_n|^2dx}{\int_{\R^N}|\tilde{u}_n|dx}\rightarrow y\in\Lambda.$$
The proof is completed.
\end{proof}

\subsection{Proof of Theorem\ref{Thm1.2} }
\begin{lemma}\label{Lem4.9}
Assume that $(V)$ and $(f_1)-(f_4)$ hold. Then, for any $\delta>0$ there exists $\varepsilon_{\delta}>0$ such that the problem \eqref{eq1.1} has at least $cat_{\Lambda_{\delta}}(\Lambda)$ nontrivial nonnegative solutions for all $\varepsilon\in(0,\varepsilon_{\delta}).$
\end{lemma}
\begin{proof}
By Lemma \ref{lem4.1} and the define of $\psi_\varepsilon,$ we have
$$\lim\limits_{\varepsilon\rightarrow0}\psi_\varepsilon\big(n_\varepsilon^{-1}(\Phi_\varepsilon(y))\big)=\lim\limits_{\varepsilon\rightarrow0}J_\varepsilon(\Phi_\varepsilon(y))=m_{_{V_0}}~\text{uniformly in}~y\in\Lambda$$
Then, there exists $\varepsilon_1>0$ such that $\tilde{\mathcal{S}}_\varepsilon:=\{u\in\mathcal{S}_\varepsilon:\psi_\varepsilon(u)\leq m_{_{V_0}}+h(\varepsilon)\}\neq0$ for all $\varepsilon\in(0,\varepsilon_1).$\\
Applying Lemma \ref{lem4.1}, Lemma \ref{lem3.3.1}, Lemma \ref{lem4.2} and Lemma \ref{lem4.5}, we can find some $\varepsilon_1=\varepsilon_\delta>0$ such that the following diagram
$$\Lambda\stackrel{\Phi_{\varepsilon}}{\rightarrow}\tilde{\mathcal{N}}_\varepsilon\stackrel{n^{-1}_{\varepsilon}}{\rightarrow}\tilde{\mathcal{S}}_\varepsilon\stackrel{n_{\varepsilon}}{\rightarrow}\tilde{\mathcal{N}}_\varepsilon\stackrel{\beta_{\varepsilon}}{\rightarrow}\Lambda_\delta$$
is well defined for any $\varepsilon\in(0,\varepsilon_1).$By the proof of \cite[Theorem5.1,Theorem5.2]{Ambrosio2017Multiplicity}, we know that for $\varepsilon>0$ small enough, we deduce from Lemma \ref{lem3.11} that $\psi_\varepsilon$ satisfies the $PS$ condition in $\tilde{\mathcal{S}}_\varepsilon.$ And $\psi_\varepsilon$ has at least $cat_{\tilde{\mathcal{S}}_\varepsilon}(\tilde{\mathcal{S}}_\varepsilon)$ critical points on $\tilde{\mathcal{S}}_\varepsilon.$ By Lemma \ref{lem3.3.1} we conclude that $J_\varepsilon$ admits at least $cat_{\Lambda_\delta}(\Lambda)$ critical points on $\mathcal{N}_\varepsilon$.

\end{proof}
\par
Now, we use a Moser iteration argument \cite{Moser} to study of behavior of the maximum points of the solutions.
\begin{lemma}\label{Lem4.10}
Let $\varepsilon_n\rightarrow0$ and $u_n\in\tilde{\mathcal{N}}_{\varepsilon_n}$ is a nontrivial nonnegative solution to \eqref{eq2.3}. Then exists $y_n\in\R^N$ such that $v_n=u_n(\cdot+y_n)$ satisfies the following problem
\begin{equation}\label{eq4.5}
\left\{ \begin{array}{ll}(-\Delta)^sv_n+V_n(x)v_n=(\frac{1}{|x|^{\mu}}\ast F(v_n))f(v_n)+|v_n|^{2^*_s-2} \,\,\ \,\,\ in \,\,\ \R^N\\
v_n\in H^s(\R^N)\\
v_n\geq 0 \,\,\ \,\,\ \,\,\ \,\,\ \,\,\,\ \,\,\,\ \,\,\,\ \,\,\ \,\,\ \,\,\ \,\,\ \,\,\ \,\,\ \,\,\ \,\,\ \,\,\ \,\,\ \,\,\ \,\,\ \,\,\ \,\,\,\ \,\,\,\ \,\,\,\ \,\,\ \,\,\ \,\,\ \,\,\ \,\,\ \,\,\ \,\,\ \,\ in \,\,\ \R^N
\end{array} \right.
\end{equation}
where $V_n(x)=V(\varepsilon_nx+\varepsilon_ny_n),$ $\varepsilon_ny_n\rightarrow y\in\Lambda$ and there exists $C>0$ such that $\|v_n\|_{L^{\infty}(\R^N)}\leq C$ for all $n\in \N.$ Furthermore, $$\lim\limits_{|x|\rightarrow\infty}v_n(x)=0 \,\,\ \text{ uniformly in } \,\,\ n\in \N.$$
\end{lemma}
\begin{proof}
For any $L>0$ and $\beta>1,$ let us define the function
$$r(v_n)=r_{L,\beta}(v_n)=v_nv_{L,n}^{2(\beta-1)}\in H^s(\R^N)$$
where $v_{L,n}=\min\{v_n,L\}.$ Since $r$ is an increasing function in $(0,+\infty),$ then we have $$(a-b)(r(a)-r(b))\geq0 \,\,\ \text{ for any } \,\,\ a,b\in\R^+.$$
Define the functions $$H(t)=\frac{|t|^2}{2} \,\,\ \text{ and } \,\,\ L(t)=\int^t_0(r'(\tau))^{\frac{1}{2}}d\tau.$$
For all $a,b\in\R$ such that $a>b,$ by applying Jensen inequality we get
\begin{equation*}
\begin{split}
H'(a-b)(r(a)-r(b))&=(a-b)(r(a)-r(b))=(a-b)\int^a_br'(t)dt\\
&=(a-b)\int^a_b(L'(t))^2dt\geq(\int^a_bL'(t)dt)^2.
\end{split}
\end{equation*}
In similar way, we can prove that the above inequality is true for all $a\leq b.$ Therefore
\begin{equation}\label{eq4.6}
H'(a-b)(r(a)-r(b))\geq|L(a)-L(b)|^2 \,\,\ \text{ for any } \,\,\ a,b\in\R.
\end{equation}
By using \eqref{eq4.6}, we have
\begin{equation}\label{eq4.7}
|L(v_n)(x)-L(v_n)(y)|^2\leq(v_n(x)-v_n(y))((v_nv_{L,n}^{2(\beta-1)})(x)-(v_nv_{L,n}^{2(\beta-1)})(y)).
\end{equation}
Now, we take $r(v_n)=v_nv_{L,n}^{2(\beta-1)}$ as test-function in \eqref{eq4.5} and in view of \eqref{eq4.7}, we obtain
\begin{equation}\label{eq4.8}
\begin{split}
&[L(v_n)]^2+\int_{\R^N}V_n(x)|v_n|^2v^{2\beta-1}_{L,n}dx\\
\leq&\int\int_{\R^{2N}}\frac{v_n(x)-v_n(y)}{|x-y|^{N+2s}}((v_nv_{L,n}^{2(\beta-1)})(x)-(v_nv_{L,n}^{2(\beta-1)})(y))dxdy+\int_{\R^N}V_{n}(x)|v_n|^2v^{2(\beta-1)}_{L,n}dx\\
=&\int_{\R^N}(\frac{1}{|x|^{\mu}}\ast F(v_n))f(v_n)v_nv^{2(\beta-1)}_{L,n}dx+\int_{\R^N}|v_n|^{2^*_s-2}v_nv^{2(\beta-1)}_{L,n}dx.
\end{split}
\end{equation}
Since $$L(v_n)\geq\frac{1}{\beta}v_nv^{2(\beta-1)}_{L,n}$$
and we can use Lemma \ref{Lem2.1} to deduce that
\begin{equation}\label{eq4.9}
[L(v_n)]^2\geq C\|L(v_n)\|^2_{L^{2^*_s}(\R^N)}\geq(\frac{1}{\beta})^2C\|v_nv^{(\beta-1)}_{L,n}\|^2_{L^{2^*_s}(\R^N)}.
\end{equation}
On the other hand, since $\{v_n\}$ is bounded in $H^s(\R^N),$ there exists $C_0>0$ such that
\begin{equation}\label{eq4.10}
\|\frac{1}{|x|^{\mu}}\ast F(v_n)\|_{L^{\infty}(\R^N)}<C_0.
\end{equation}
Taking $\xi\in(0,V_0),$ and using \eqref{eq3.1}, \eqref{eq4.9} and \eqref{eq4.10}, we can see that \eqref{eq4.8} yields $$\|v_nv^{\beta-1}_{L,n}\|^2_{L^{2^*_s}(\R^N)}\leq C\beta^2 \big(\int_{\R^N}|v_n|^qv^{2(\beta-1)}_{L,n}dx+\int_{\R^N}|v_n|^{2^*_s}v_{L,n}^{2(\beta-1)}dx\big).$$
Set $q+2\beta-2=2^*_s$ $\Rightarrow$ $\beta=\frac{1}{2}(2^*_s+2-q)>1,$ then
\begin{equation*}
\begin{split}
&\big(\int_{\R^N}|v_nv_{L,n}^{\beta-1}|^{2^*_s}dx\big)^{\frac{2}{2^*_s}}\\
\leq&C\beta^2\big(\int_{\R^N}|v_n|^qv_{L,n}^{2\beta-1}dx+\int_{\R^N}|v_n|^{2^*_s-1}(v_nv_{L,n}^{2(\beta-1)})dx\big)\\
\leq&C\beta^2\big(\int_{\R^N}|v_n|^{2^*_s}dx+\int_{\{v_n\leq R_0\}}|v_n|^{2^*_s-1}(v_nv_{L,n}^{2(\beta-1)})dx+\int_{\{v_n>R_0\}}|v_n|^{2^*_s-2}(v_nv_{L,n}^{\beta-1})^2dx\big).
\end{split}
\end{equation*}
By $\{u_n\}$ is bounded in $H_s$, $\exists R_0>0$ $s.t.$ $\big(\int_{\{v_n>R_0\}}|v_n|^{2^*_s}dx\big)^\frac{2^*_s-2}{2^*_s}\leq\frac{1}{2C\beta^2}.$ Hence, we can see that
\begin{equation*}
\begin{split}
&\int_{\{v_n\leq R_0\}}|v_n|^{2^*_s-q+1}|v_n|^{q-1}(v_nv{L,n}^{2(\beta-1)})dx+\big(\int_{\{v_n>R_0\}}|v_n|^{2^*_s}dx\big)^{\frac{2^*_s-2}{2^*_s}}\big(\int_{\{v_n>R_0\}}(v_nv_{L,n}^{\beta-1})^{2^*_s}dx\big)^{\frac{2}{2^*_s}}\\
\leq&R_0^{2^*_s-q+1}\int_{\R^N}|v_n|^{2^*_s}dx+\frac{1}{2C\beta^2}\big(\int_{\R^N}(v_nv_{L,n}^{\beta-1})^{2^*_s}dx\big)^{\frac{2}{2^*_s}}.
\end{split}
\end{equation*}
Therefore, we can deduce that
\begin{equation}\label{eq4.11}
\big(\int_{\R^N}|v_nv_{L,n}^{\beta-1}|^{2^*_s}dx\big)^{\frac{2}{2^*_s}}\leq2C\beta^2(1+R_0^{2^*_s-q+1})\int_{\R^N}|v_n|^{2^*_s}dx<C<+\infty.
\end{equation}
Taking the limit in \eqref{eq4.11} as $L\rightarrow+\infty$ and Fatou lemma, we have $\big(\int_{\R^N}|v_n|^{2^*_s\beta}dx\big)^{\frac{2}{2^*_s}}\leq C<+\infty.$ so, $v_n\in L^{2^*_s\beta}(\R^N).$ For any $\beta>\frac{1}{2}(2^*_s+2-q)>1$ and $\beta\leq1+\frac{2^*_s}{2}\cdot\frac{2^*_s-q}{2}$ then $2<q+2\beta-2<2^*_s+2\beta-2\leq2^*_s(1+\frac{2^*_s-q}{2}).$ we can deduce that
\begin{equation*}
\big(\int_{\R^N}|v_n|^{2^*_s\beta}dx\big)^{\frac{2}{2^*_s}}\leq C\beta^2\big(\int_{\R^N}|v_n|^{q+2\beta-2}dx+\int_{\R^N}|v_n|^{2^*_s+2\beta-2}dx\big)\leq C_0<+\infty.
\end{equation*}
Let $a=\frac{2^*_s(2^*_s-q)}{2(\beta-1)},$ $b=q+2\beta-2-a,$ $r=\frac{2^*_s}{a},$ $r'=\frac{2^*_s}{2^*_s-a},$ then $\frac{2^*_sb}{2^*_s-a}=2^*_s+2\beta-2.$ Taking into account Young inequality we have
$$\int_{\R^N}|v_n|^{q+2\beta-2}dx\leq\frac{a}{2^*_s}\int_{\R^N}|v_n|^{2^*_s}dx+\frac{2^*_s-a}{2^*_s}\int_{\R^N}|v_n|^{2^*_s+2\beta-2}dx\leq C\big(1+\int_{\R^N}|v_n|^{2^*_s+2\beta-2}dx\big).$$
Therefore, $$\big(\int_{\R^N}|v_n|^{2^*_s\beta}dx\big)^{\frac{2}{2^*_s}}\leq C\beta^2\big(1+\int_{\R^N}|v_n|^{2^*_s+2\beta-2}dx\big).$$
We note to $\beta>1,$ we deduce that
\begin{equation}\label{eq4.12}
\big(1+\int_{\R^N}|v_n|^{2^*_s\beta}dx\big)^{\frac{2}{2^*_s}}\leq C\beta^2\big(1+\int_{\R^N}|v_n|^{2^*_s+\beta-2}dx\big).
\end{equation}
Now, we set $\beta=1+\frac{2^*_s}{2}\cdot\frac{2^*_s-q}{2},$ then observing that $2^*_s+2\beta-2=2^*_s(\frac{1}{2}(2^*_s+2-q)).$ Iterating this process and recalling that $2^*_s+2\beta_{i-1}-2=2^*_s\beta_i.$
Argue as \cite{He2016Existence}. Thus, $$\beta_{i+1}-1=(\frac{2^*_s}{2})^i(\beta_1-1).$$
Replacing it in \eqref{eq4.12} we have
\begin{equation*}
\big(1+\int_{\R^N}|v_n|^{2^*_s\beta_{i+1}}dx\big)^{\frac{1}{2^*_s(\beta_{i+1}-1)}}\leq(C\beta_{i+1}^2)^{\frac{1}{2(\beta_{i+1}-1)}}\big(1+\int_{\R^N}v_n^{2\beta_i+2^*_s-2}dx\big)^{\frac{1}{2(\beta_i-1)}}.
\end{equation*}
Denoting $C_{i+1}=C\beta_{i+1}^2$ and $K_i:=(1+\int_{\R^N}v_n^{2\beta_i+2^*_s-2}dx)^{\frac{1}{2(\beta_i-1)}}.$ We conclude that there exists a constant $C_0>0$ independent of $i,$ such that $$K_{i+1}\leq\prod^{i+1}_{i=2}C_i^{\frac{1}{2(\beta_i-1)}}K_1\leq CK_1.$$
Therefore, $$\|v_n(x)\|_{L^{\infty}}(\R^N)\leq C_0K_1<\infty,$$
uniformly on $n\in\N,$ thanks to $v_n\in L^{2^*_s\beta_1}(\R^N)$ and $\|v_n\|_{\varepsilon_n}\leq C.$
Arguing as in \cite{Alves2016}, we can prove that $$\lim\limits_{|x|\rightarrow\infty}v_n(x)=0 \text{    uniformly in  } n\in \N.$$
\end{proof}
\par
Now we consider $\varepsilon_n\rightarrow0^+$ and take a sequence $u_n\in \tilde{\mathcal{N}}_\varepsilon$ of solutions of the problem \eqref{eq2.3} as above. There exists $\gamma>0$ such that
\begin{equation}\label{eq4.13}
\|u_n\|_{L^\infty(\R^N)}\geq\gamma \,\,\,\,\ \text{ uniformly in } \,\,\,\ n\in \N.
\end{equation}
Assume by contradiction, we have $\lim\limits_{n\rightarrow\infty}\|u_n\|_{L^\infty(\R^N)}=0.$ For any $\xi>0,$ there exists $n_0$ such that $\|u_n\|_{L^\infty(\R^N)}<\xi$ for any $n>n_0.$
Since $u_n\in\tilde{\mathcal{N}}_\varepsilon,$ we have
\begin{equation*}
\begin{split}
\|u_n\|_{\varepsilon_n}^2=&\int_{\R^N}\big(\frac{1}{|x|^\mu}\ast F(u_n)\big)f(u_n)u_ndx+\int_{\R^N}|u_n|^{2^*_s}dx\\
\leq&C\big(\int_{\R^N}(|u_n|^{2t}+|u_n|^{qt})dx\big)^{\frac{2}{t}}+\int_{\R^N}|u_n|^{2^*_s}dx
\end{split}
\end{equation*}
where $t=\frac{2N}{2N-\mu}.$ Since $2t\in(2,2^*_s)$ and $qt\in(2,2^*_s),$ there exists $\sigma>0$ small enough such that $(2t-\sigma)\in(2,2^*_s)$ and $(qt-\sigma)\in(2,2^*_s).$ Since we have that $\{u_n\}$ is bound in $H_0^1(\R^N),$ we can deduce to
\begin{equation*}
\begin{split}
\|u_n\|_{\varepsilon_n}&\leq C\big(\int_{\R^N}(|u_n|^{2t-\sigma}|u_n|^\sigma+|u_n|^{qt-\sigma}|u_n|^\sigma)dx\big)^{\frac{2}{t}}+\int_{\R^N}|u_n|^{2^*_s-\sigma}|u_n|^{\sigma}dx\\
&\leq C\|u_n\|_{L^\infty(\R^N)}^{\sigma\cdot\frac{2}{t}}\big(\int_{\R^N}(|u_n|^{2t-\sigma}+|u_n|^{qt-\sigma})dx\big)^{\frac{2}{t}}+\|u_n\|_{L^\infty(\R^N)}^{\sigma}\int_{\R^N}|u_n|^{2^*_s-\sigma}dx\\
&<C_1\xi^{\frac{2\sigma}{t}}+C_2\xi^{\sigma}.
\end{split}
\end{equation*}
This implies that $\|u_n\|_{\varepsilon_n}\rightarrow 0~(n\rightarrow\infty).$ In similar way, we can decude
$$\frac{1}{2}\int_{\R^N}\big(\frac{1}{|x|^\mu}\ast F(u_n)\big)F(u_n)dx+\frac{1}{2^*_s}\int_{\R^N}|u_n|^{2^*_s}dx\rightarrow0~(n\rightarrow\infty),$$
then $J_{\varepsilon_n}(u_n)\rightarrow0~(n\rightarrow\infty),$ this contradict with $J_{\varepsilon_n}(u_n)\rightarrow m_{_{V_0}}>0.$ As a consequence, \eqref{eq4.13} holds.
By Lemma \ref{Lem4.10}, we have
$$\|v_n\|_{L^\infty(\R^N)}\leq C, \,\,\,\ \text{uniformly in} \,\,\,\ n\in \N,$$
and
$$\lim\limits_{|x|\rightarrow\infty}v_n(x)=0 \,\,\,\ \text{uniformly in} \,\,\,\ n\in \N.$$
There exists $R>0$ such that  $\|v_n\|_{L^\infty(B^c_R(0))}<\gamma,$
then
\begin{equation}\label{eq5.14}
\|u_n\|_{L^\infty(B^c_R(y_n))}<\gamma.
\end{equation}
Hence
\begin{equation}\label{eq5.15}
\|u_n\|_{L^\infty(B_R(y_n))}\geq\gamma.
\end{equation}
Let $p_n$ is the global maximum point of $u_n,$ taking into account \eqref{eq5.14} and \eqref{eq5.15} we can get $p_n\in B_R(y_n).$ Hence, $p_n=y_n+q_n$ for some $q_n\in B_R(0).$ Then $\xi_{\varepsilon_n}=\varepsilon_ny_n+\varepsilon_nq_n$ is the maximum point of $u_n(\frac{x}{\varepsilon_n}).$ Since $|q_n|<R$ for any $n\in\N$ and $\varepsilon_ny_n\rightarrow y_0\in\Lambda.$ Therefore,
$$\lim\limits_{n\rightarrow\infty}V(\xi_{\varepsilon_n})=V(y_0)=V_0.$$
which ends the proof of the Theorem\ref{Thm1.2}.

\raisebox{-0.5mm}{\rule{1.5mm}{1mm}}\vspace{6pt}
\section*{Acknowledgment}
We would like to thank the anonymous referee for his/her careful readings of our manuscript and the useful comments made for its improvement.


\begin{thebibliography}{10}

\bibitem{Ackermann04MZ}
N.~Ackermann.
\newblock On a periodic {S}chr\"{o}dinger equation with nonlocal superlinear
  part.
\newblock {\em Math. Z.}, 248(2):423--443, 2004.

\bibitem{Alves2014Multiplicity}
C.O. Alves, P.C. Carriao, and E.S. Medeiros.
\newblock Multiplicity of solutions for a class of quasilinear problem in
  exterior domains with neumann conditions.
\newblock {\em Abstr. Appl. Anal.}, 2004(3):251, 2014.

\bibitem{Alves-Yang17JDE}
C.O. Alves, F.~Gao, M.~Squassina, and M.~Yang.
\newblock Singularly perturbed critical {C}hoquard equations.
\newblock {\em J. Differential Equations}, 263(7):3943--3988, 2017.

\bibitem{Alves2016}
C.O. Alves and O.H. Miyagaki.
\newblock Existence and concentration of solution for a class of fractional
  elliptic equation in {$\Bbb R^N$} via penalization method.
\newblock {\em Calc. Var. Partial Differential Equations}, 55(3):19, 2016.

\bibitem{Alves-Yang16PRSE}
C.O. Alves and M.~Yang.
\newblock Investigating the multiplicity and concentration behaviour of
  solutions for a quasi-linear {C}hoquard equation via the penalization method.
\newblock {\em Proc. Roy. Soc. Edinburgh Sect. A}, 146(1):23--58, 2016.

\bibitem{Ambrosetti-Malchiodi07CM}
A.~Ambrosetti and A.~Malchiodi.
\newblock Concentration phenomena for nonlinear {S}chr\"{o}dinger equations:
  recent results and new perspectives.
\newblock In {\em Perspectives in nonlinear partial differential equations},
  volume 446 of {\em Contemp. Math.}, pages 19--30. Amer. Math. Soc.,
  Providence, RI, 2007.

\bibitem{Ambrosio2017Multiplicity}
Vincenzo Ambrosio.
\newblock Multiplicity and concentration results for a fractional choquard
  equation via penalization method.
\newblock {\em Potential Analysis}, (1):1--28, 2017.

\bibitem{Miyagaki17NA}
P.~Belchior, H.~Bueno, O.~H. Miyagaki, and G.~A. Pereira.
\newblock Remarks about a fractional {C}hoquard equation: ground state,
  regularity and polynomial decay.
\newblock {\em Nonlinear Anal.}, 164:38--53, 2017.

\bibitem{Bhattarai17JDE}
S.~Bhattarai.
\newblock On fractional {S}chr\"{o}dinger systems of {C}hoquard type.
\newblock {\em J. Differential Equations}, 263(6):3197--3229, 2017.

\bibitem{Brezis-Lieb1983PAMS}
H.~Br\'ezis and E.~Lieb.
\newblock A relation between pointwise convergence of functions and convergence
  of functionals.
\newblock {\em Proc. Amer. Math. Soc.}, 88:486--490, 1983.

\bibitem{Valdinocibook}
C.~Bucur and E.~Valdinoci.
\newblock {\em Nonlocal diffusion and applications}, volume~20 of {\em Lecture
  Notes of the Unione Matematica Italiana}.
\newblock Springer, 2016.

\bibitem{Cassani-Zhang18ANA}
D.~Cassani and J.~Zhang.
\newblock Choquard-type equations with {H}ardy-{L}ittlewood-{S}obolev
  upper-critical growth.
\newblock {\em Adv. Nonlinear Anal.}, 2018.

\bibitem{Chen-Liu16Non}
Y.~Chen and C.~Liu.
\newblock Ground state solutions for non-autonomous fractional {C}hoquard
  equations.
\newblock {\em Nonlinearity}, 29(6):1827--1842, 2016.

\bibitem{Cotsiolis-Tavoularis04JMAA}
A.~Cotsiolis and N.~Tavoularis.
\newblock Best constants for {S}obolev inequalities for higher order fractional
  derivatives.
\newblock {\em J. Math. Anal. Appl.}, 295(1):225--236, 2004.

\bibitem{d'AveniaMMMAS15}
P.~d'Avenia, G.~Siciliano, and M.~Squassina.
\newblock On fractional {C}hoquard equations.
\newblock {\em Math. Models Methods Appl. Sci.}, 25(8):1447--1476, 2015.

\bibitem{Pino1996CVPDE}
M.~del Pino and P.L. Felmer.
\newblock Local mountain passes for semilinear elliptic problems in unbounded
  domains.
\newblock {\em Calc. Var. Partial Differential Equations}, 4(2):121--137, 1996.

\bibitem{Nezza2012BDSM}
E.~Di~Nezza, G.~Palatucci, and E.~Valdinoci.
\newblock Hitchhiker's guide to the fractional {S}obolev spaces.
\newblock {\em Bull. Sci. Math.}, 136(5):521--573, 2012.

\bibitem{Gao-Yang18SCI}
F.~Gao and M.~Yang.
\newblock The {B}rezis-{N}irenberg type critical problem for the nonlinear
  {C}hoquard equation.
\newblock {\em Sci. China Math.}, 61(7):1219--1242, 2018.

\bibitem{Gao-Tang-Chen18ZMAP}
Z.~Gao, X.~Tang, and S.~Chen.
\newblock On existence and concentration behavior of positive ground state
  solutions for a class of fractional {S}chr\"{o}dinger--{C}hoquard equations.
\newblock {\em Z. Angew. Math. Phys.}, 69(5):69:122, 2018.

\bibitem{Guo-Hu18MMAS}
L.~Guo and T.~Hu.
\newblock Existence and asymptotic behavior of the least energy solutions for
  fractional {C}hoquard equations with potential well.
\newblock {\em Math. Methods Appl. Sci.}, 41(3):1145--1161, 2018.

\bibitem{He2016Existence}
Xiaoming He and Wenming Zou.
\newblock Existence and concentration result for the fractional
  {S}chr\"{o}dinger equations with critical nonlinearities.
\newblock {\em Calculus of Variations and Partial Differential Equations},
  55(4):91, 2016.

\bibitem{Laskin2000Physics}
N.~Laskin.
\newblock Fractional quantum mechanics and {L}\'evy path integrals.
\newblock {\em Phys. Lett. A}, 268(4-6):298--305, 2000.

\bibitem{Laskin2002physics}
N.~Laskin.
\newblock Fractional {S}chr\"odinger equation.
\newblock {\em Phys. Rev. E (3)}, 66(5):056108, 2002.

\bibitem{Lieb1976SAM}
E.H. Lieb.
\newblock Existence and uniqueness of the minimizing solution of {C}hoquard's
  nonlinear equation.
\newblock {\em Studies in Appl. Math.}, 57(2):93--105, 1976/77.

\bibitem{Lions80NA}
P.-L. Lions.
\newblock The {C}hoquard equation and related questions.
\newblock {\em Nonlinear Anal.}, 4(6):1063--1072, 1980.

\bibitem{Ma-Zhao10ARMA}
L.~Ma and L.~Zhao.
\newblock Classification of positive solitary solutions of the nonlinear
  {C}hoquard equation.
\newblock {\em Arch. Ration. Mech. Anal.}, 195(2):455--467, 2010.

\bibitem{Ma-Zhang17NA}
P.~Ma and J.~Zhang.
\newblock Existence and multiplicity of solutions for fractional {C}hoquard
  equations.
\newblock {\em Nonlinear Anal.}, 164:100--117, 2017.

\bibitem{Moroz-Schaftingen13JFA}
V.~Moroz and J.~Van~Schaftingen.
\newblock Groundstates of nonlinear {C}hoquard equations: existence,
  qualitative properties and decay asymptotics.
\newblock {\em J. Funct. Anal.}, 265(2):153--184, 2013.

\bibitem{Moroz-Schaftingen15TAMC}
V.~Moroz and J.~Van~Schaftingen.
\newblock Existence of groundstates for a class of nonlinear {C}hoquard
  equations.
\newblock {\em Trans. Amer. Math. Soc.}, 367(9):6557--6579, 2015.

\bibitem{Moroz-Schaftingen15CVPDE}
V.~Moroz and J.~Van~Schaftingen.
\newblock Semi-classical states for the {C}hoquard equation.
\newblock {\em Calc. Var. Partial Differential Equations}, 52(1-2):199--235,
  2015.

\bibitem{Moser}
J.~Moser.
\newblock A new proof of {D}e {G}iorgi's theorem concerning the regularity
  problem for elliptic differential equations.
\newblock {\em Comm. Pure Appl. Math.}, 13:457--468, 1960.

\bibitem{Mukheriee2016}
T.~Mukheriee and K.~Sreenadh.
\newblock Existence and multiplicity results for brezis-nirenberg type
  fractional choquard equation.
\newblock {\em math.AP}, 2016.

\bibitem{Sreenadh17NODEA}
T.~Mukherjee and K.~Sreenadh.
\newblock Fractional {C}hoquard equation with critical nonlinearities.
\newblock {\em NoDEA Nonlinear Differential Equations Appl.}, 24(6):Art. 63,
  34, 2017.

\bibitem{Palatucci}
G.~Palatucci and A.~Pisante.
\newblock Improved {S}obolev embeddings, profile decomposition, and
  concentration-compactness for fractional {S}obolev spaces.
\newblock {\em Calc. Var. Partial Differential Equations}, 50(3-4):799--829,
  2014.

\bibitem{Pekar}
S.~Pekar.
\newblock {\em Untersuchung \"{u}ber die Elektronentheorie der Kristalle}.
\newblock Akademie Verlag, Berlin, 1954.

\bibitem{Rabinowitz1992}
P.H. Rabinowitz.
\newblock On a class of nonlinear {S}chr\"odinger equations.
\newblock {\em Z. Angew. Math. Phys.}, 43(2):270--291, 1992.

\bibitem{Shen-Gao-YangMMAS16}
Z.~Shen, F.~Gao, and M.~Yang.
\newblock Ground states for nonlinear fractional {C}hoquard equations with
  general nonlinearities.
\newblock {\em Math. Methods Appl. Sci.}, 39(14):4082--4098, 2016.

\bibitem{Szulkin2010}
A.~Szulkin and T.~Weth.
\newblock {\em The method of {N}ehari manifold}.
\newblock Int. Press, Somerville, MA, 2010.

\bibitem{Wang-Xiang16EJDE}
F.~Wang and M.~Xiang.
\newblock Multiplicity of solutions to a nonlocal {C}hoquard equation involving
  fractional magnetic operators and critical exponent.
\newblock {\em Electron. J. Differential Equations}, pages Paper No. 306, 11,
  2016.

\bibitem{Wang-Yang18BVP}
Y.~Wang and Y.~Yang.
\newblock Bifurcation results for the critical {C}hoquard problem involving
  fractional p-{L}aplacian operator.
\newblock {\em Bound. Value Probl.}, page 2018:132, 2018.

\bibitem{Wei-Winter09JMP}
J.~Wei and M.~Winter.
\newblock Strongly interacting bumps for the {S}chr\"{o}dinger-{N}ewton
  equations.
\newblock {\em J. Math. Phys.}, 50(1):012905, 22, 2009.

\bibitem{yang}
Z.~Yang.
\newblock Multiplicity and concentration behaviour of solutions for a
  fractional choquard equation with critical growth.
\newblock {\em In press.}

\bibitem{Yang-Zhao18AML}
Z.~Yang and F.~Zhao.
\newblock Three solutions for a fractional {S}chr\"odinger equation with
  vanishing potentials.
\newblock {\em Appl. Math. Lett.}, 76:90--95, 2018.

\bibitem{Zhang-Wang-Zhang19CPAA}
H.~Zhang, J.~Wang, and F.~Zhang.
\newblock Semiclassical states for fractional {C}hoquard equations with
  critical growth.
\newblock {\em Commun. Pure Appl. Anal.}, 18(1):519--538, 2019.

\bibitem{Zhang2015Solutions}
Jianjun Zhang and Wenming Zou.
\newblock Solutions concentrating around the saddle points of the potential for
  critical {S}chr\"{o}dinger equations.
\newblock {\em Calc. Var. Partial Differential Equations}, 54(4):4119--4142,
  2015.

\bibitem{Zhang-Wu18JMAA}
W.~Zhang and X.~Wu.
\newblock Nodal solutions for a fractional {C}hoquard equation.
\newblock {\em J. Math. Anal. Appl.}, 464(2):1167--1183, 2018.

\bibitem{Zhongbook}
C.~Zhong, X.~Fan, and W.~Chen.
\newblock {\em Introduction of nonlinear functional analysis}.
\newblock Lanzhou University Publishing House, 1998.

\end{thebibliography}
\end{document}